\documentclass[10pt]{amsart}
\usepackage{bm}

\usepackage{amsmath}
\usepackage{amsthm}
\usepackage{amssymb}

\usepackage{hyperref}

\newcommand{\R}{\mathbb{R}}
\newcommand{\C}{\mathbb{C}}
\newcommand{\N}{\mathbb{N}}
\newcommand{\Z}{\mathbb{Z}}
\newcommand{\lr}[1]{\langle #1 \rangle}

\newcommand{\eps}{\varepsilon}
\newcommand{\wt}[1]{\widetilde{#1}}
\newcommand{\wh}[1]{\widehat{#1}}

\newcommand{\Sc}{\mathcal{S}}
\newcommand{\Lc}{\mathcal{L}}
\newcommand{\J}{\mathcal{J}}

\newcommand{\wS}{\Sigma}

\newcommand{\lp}{\mathcal{U}}
\newcommand{\self}{\mathfrak{u}}

\newcommand{\dea}{\mathfrak{a}}

\newcommand{\ol}[1]{{\overline{#1}}}

\newcommand{\uhp}{u^{\rm{hyp},+}}

\newcommand{\uhpn}{u^{\rm{hyp},\pm}}

\newcommand{\ue}{u^{\rm{ell}}}
\newcommand{\uep}{u^{\rm{ell},+}}

\newcommand{\uel}{u_{\le t^{-\frac{1}{\m}}}}

\newcommand{\PN}{P_{\frac{N}{2} \le \cdot \le 2N}}

\newcommand{\PNv}{P_{\frac{N_v}{2} \le \cdot \le 2N_v}}

\newcommand{\dx}{\partial_x}
\newcommand{\dy}{\partial_y}
\newcommand{\dt}{\partial_t}

\newcommand{\dz}{\partial_{\zeta}}
\newcommand{\de}{\partial_{\eta}}

\newcommand{\err}{\bm{\mathrm{err}}}

\newcommand{\Ot}{\Omega (t)}
\newcommand{\Orz}{\mathfrak{Y} (t)}

\newcommand{\Xp}{\mathfrak{X}^+ (t)}
\newcommand{\Xn}{\mathfrak{X}^- (t)}

\newcommand{\Xz}{\mathfrak{X}^0 (t)}
\newcommand{\Xf}{\wh{\mathfrak{X}}^+ (t)}

\newcommand{\m}{\mathfrak{m}}
\newcommand{\p}{\mathfrak{p}}
\newcommand{\xc}{\bm{X}}
\newcommand{\con}{\mathfrak{C}}

\newcommand{\al}{\alpha}

\newcommand{\kk}{\kappa}

\DeclareMathOperator{\supp}{supp}

\newtheorem{thm}{Theorem}[section]
\newtheorem{prop}[thm]{Proposition}

\newtheorem{lem}[thm]{Lemma}
\newtheorem{cor}[thm]{Corollary}

\theoremstyle{remark}
\newtheorem{rmk}[thm]{Remark}

\numberwithin{equation}{section}

\title[Asymptotic behavior for a KdV-type equation]{Asymptotic behavior of solutions to \\ a higher-order KdV-type equation \\ with critical nonlinearity}

\author[M. Okamoto]{Mamoru Okamoto}
\address{Division of Mathematics and Physics, Faculty of Engineering, Shinshu University, 4-17-1 Wakasato, Nagano City 380-8553, Japan}
\email{m\_okamoto@shinshu-u.ac.jp}

\subjclass[2010]{35Q53, 35B40}

\keywords{higher-order KdV-type equation, asymptotic behavior, critical nonlinearity, self-similar solution}

\date{\today}

\allowdisplaybreaks[1]
\begin{document}

\begin{abstract}
We consider the Cauchy problem of the higher-order KdV-type equation:
\[
\dt u + \frac{1}{\m} |\dx|^{\m-1} \dx u =  \dx (u^{\m})
\]
where $\m \ge 4$.
The nonlinearity is critical in the sense of long-time behavior.
Using the method of testing by wave packets, we prove that there exists a unique global solution of the Cauchy problem satisfying the same time decay estimate as that of linear solutions.
Moreover, we divide the long-time behavior of the solution into three distinct regions.
\end{abstract}

\maketitle

\section{Introduction}

We consider the Cauchy problem for the higher-order Korteweg-de Vries (KdV) type equation
\begin{equation} \label{hKdV}
\dt u + \frac{1}{\m} |\dx|^{\m-1} \dx u =  \dx (u^{\m}),
\end{equation}
where $u$ is a real valued function, $|\dx| = (-\dx^2)^{\frac{1}{2}}$, and $\m \in \Z_{\ge 3}$.
This equation describes the propagation of nonlinear waves in a dispersive medium.
In particular, \eqref{hKdV} with $\m=3,5$ are called the modified KdV (mKdV) and the modified Kawahara equations, respectively.

If $u$ is a solution to \eqref{hKdV}, then the total mass, the momentum, and the energy are conserved:
\begin{align*}
&\int_{\R} u(t,x) dx, \quad
\int_{\R} u(t,x)^2dx, \\
&\int_{\R} \left\{ \frac{1}{2\m} \left| |\dx|^{\frac{\m-1}{2}} u(t,x) \right|^2 - \frac{1}{\m+1} u(t,x)^{\m+1} \right\} dx.
\end{align*}
Local-in-time well-posedness of the Cauchy problem for \eqref{hKdV} follows from the same argument as in \cite{KPV93} (see also \cite{KPV89, KPV91, Gru10} and references therein).
Owing to the conserved quantities, the local-in-time solution can be extended to the global-in-time one if the values of the initial data are small.
Thus, it is of interest to obtain the asymptotic behavior of solutions to \eqref{hKdV}.

Sidi et al. \cite{SSS86} studied the long-time behavior of solutions to the generalized KdV equations
\begin{equation} \label{gKdV}
\dt u + \frac{1}{\al} |\dx|^{\al-1} \dx u = \dx (u^p)
\end{equation}
for $\al \in \R$, $\al \ge 1$, and $p \in \Z_{\ge 2}$.
More precisely, they proved that when the initial data values are small, the global-in-time solution scatters to a linear solution if $\al \ge 1$ and $p > \frac{\al+\sqrt{\al^2+4\al}}{2}+1$.
Kenig et al. \cite{KPV91} improved the results in \cite{SSS86}, that is the scattering for small initial data values holds true if $\al \ge 1$ and $p>\max (\al+1,\frac{\al}{2}+3)$.
Because there exists a blow up solution in some cases when initial data values are large (\cite{KPV00, Mer01}), the assumption of small initial data values is essential.

The asymptotic behavior for \eqref{gKdV} with $\al=3$ has been studied by several researchers (see \cite{DeiZho93, HayNau98, HayNau99, HayNau01, Tao07, KocMar12, HayNau16, HarG16, GPR16, Dod17} and references therein).
In particular, $p=3$ is critical in the sense of long-time behavior.
In other words, while the solutions scatter for $p>3$, asymptotic behavior of the solution differs from that of the linear solutions when $p=3$.
Moreover, Hayashi and Naumkin \cite{HayNau151, HayNau153} showed the criticality of the quartic derivative fourth-order nonlinear Schr\"{o}dinger equation (see also \cite{HayNau152, HirOka16}), which is related to \eqref{hKdV} with $\m=4$.

From these results, we expect that the nonlinearity of \eqref{hKdV} is critical in the sense of long-time behavior.
However, there is a gap between the exponent of nonlinearity in previous results and that to be critical in general.
In this paper, we study the asymptotic behavior of solutions to \eqref{hKdV} for $\m \ge 4$.
Even though we used $u^{\m}$ in \eqref{hKdV} in our study, the same asymptotic behavior is obtained for \eqref{hKdV} with short-range perturbations (see Remark \ref{rmk:shp}).

To explain the critical phenomenon, we roughly derive the asymptotic behavior of linear solutions for Schwartz initial data $u_0 \in \mathcal{S} (\R)$.
Let $\lp (t)$ denote the linear propagator, i.e., $\lp (t) := e^{-\frac{1}{\m}t |\dx|^{\m-1} \dx}$.
We note that, for $t>0$, the linear solution is written as follows:
\[
\lp (t) u_0 (x)
= t^{-\frac{1}{\m}} \int_{\R} Q_0 \left( t^{-\frac{1}{\m}} (x-y) \right) u_0 (y) dy,
\]
where $Q_0$ is defined by $Q_0 (y) := \frac{1}{2\pi} \int_{\R} e^{i (y \xi - \frac{1}{\m} |\xi|^{\m-1}\xi)} d\xi$, that is, $Q_0$ satisfies the ordinary differential equation
\begin{equation*}
\left| \frac{d}{dy} \right|^{\m-1}Q_0 - y Q_0=0.
\end{equation*}
Sidi et al. \cite{SSS86} proved the following estimate for $Q_0$:
\[
\left| \frac{d^k Q_0}{dy^k} (y) \right| \lesssim \lr{y}^{\frac{k}{\m-1}-\frac{\m-2}{2(\m-1)}},
\]
where $\lr{y} := (1+|y|^2)^{\frac{1}{2}}$.

The Hamiltonian flow corresponding to the linear equation is given by
\[
(x,\xi) \mapsto (x+t|\xi|^{\m-1},\xi).
\]
We expect that, for $v>0$, solutions initially localized spatially near zero and at frequencies near $\pm \xi_v$, where $\xi_v := v^{\frac{1}{\m-1}}$, travel along the ray $\Gamma_v := \{ x=vt \}$.
Hence, the linear solution $\lp (t) u_0(x)$ decays rapidly as $t^{-\frac{1}{\m}} x \to -\infty$ and oscillates as $t^{-\frac{1}{\m}} x \to +\infty$.
In particular, the stationary phase method shows that, as $t^{-\frac{1}{\m}} x \to +\infty$, there exists a constant $c_0$ such that
\[
\lp (t) u_0 (x) = c_0 t^{-\frac{1}{\m}} (t^{-\frac{1}{\m}} x)^{-\frac{\m-2}{2(\m-1)}} \Re \left\{ \wh{u_0} \left( t^{-\frac{1}{\m-1}} x^{\frac{1}{\m-1}} \right) e^{i \phi (t,x)} \right\} + \text{error},
\]
where the phase function is given by
\begin{equation} \label{phase}
\phi (t,x) := \frac{\m-1}{\m} t^{-\frac{1}{\m-1}} |x|^{\frac{\m}{\m-1}} - \frac{\pi}{4}.
\end{equation}
Moreover, in the self-similar region $|t^{-\frac{1}{\m}} x| \lesssim 1$,
\[
\lp (t) u_0 (x) = t^{-\frac{1}{\m}} Q_0 (t^{-\frac{1}{\m}} x) \int_{\R} u_0 (y) dy + \text{error}.
\]

This observation implies that if $\| u_0 \|_{L^2}+ \| x u_0 \|_{L^2} \le \eps$ then we have
\[
\left| \dx^k \lp (t) u_0 (x) \right| \lesssim \eps t^{-\frac{k+1}{\m}} \lr{t^{-\frac{1}{\m}}x}^{\frac{k}{\m-1}-\frac{\m-2}{2(\m-1)}}
\]
for $k=0,1, \dots, \m-2$.
We expect that solutions to \eqref{hKdV} satisfy the same pointwise estimates as linear solutions above when the initial data values are small.
Then,
\[
\| u(t) \|_{L_x^2} \lesssim \| u(1) \|_{L^2} + \eps^{\m-1} \int_1^t \| u(t') \|_{L_x^2} \frac{dt'}{t'}.
\]
Because the integral is not bounded by $\sup_{t \ge 1} \| u(t) \|_{L^2}$, we only expect that the solution behaves like a linear solution up to $t \sim \exp(\eps^{-\m+1})$.
Especially, the asymptotic behavior of the solution differs from that of linear solutions.

\subsection{Main result}

To state our results precisely, we introduce notation and function spaces.
We denote the set of positive and negative real numbers by $\R_+$ and $\R_-$, respectively.
We denote the usual Sobolev space by $H^s(\R)$.
For $s \in \R$, we denote the weighted Sobolev spaces by $\wS^{s}(\R)$ equipped with the norm $\| f \|_{\wS^{s}} := \| f \|_{H^s} + \| x f \|_{L^2}$.

\begin{thm} \label{thm}
Let $\m \ge 4$.
Assume that the initial datum $u_0$ at time $0$ satisfies
\[
\| u_0 \|_{\wS^{\frac{\m-1}{2\m}}} \le \eps \ll 1.
\]
Then, there exists a unique global solution $u$ to \eqref{hKdV} with $\lp (-t) u \in C \big( \R; \wS^{\frac{\m-1}{2\m}}(\R) \big)$ satisfying the estimates
\begin{equation} \label{est:u_infty}
\left\| \lr{t^{-\frac{1}{\m}}x}^{-\frac{k}{\m-1}+\frac{\m-2}{2(\m-1)}} \dx^k u(t) \right\|_{L^{\infty}} \lesssim \eps t^{-\frac{k+1}{\m}}
\end{equation}
for $t \ge 1$ and $k=0,1, \dots, \m-2$.
Moreover, we have the following asymptotic behavior as $t \to +\infty$.

In the decaying region $\Xn := \{ x \in \R_- \colon t^{-\frac{1}{\m}} |x| \gtrsim t^{(\m-1)\rho} \}$, where $\rho := \frac{1}{\m} (\frac{1}{2\m}-\eps)$, we have
\[
\left\| t^{\frac{1}{\m}} \lr{t^{-\frac{1}{\m}} x}^{\frac{2\m-3}{2(\m-1)}} u \right\|_{L^{\infty} (\Xn)} + \left\| t^{\frac{1}{2\m}} \lr{t^{-\frac{1}{\m}} x} u \right\|_{L^2 (\Xn)} \lesssim \eps.
\]
In the self-similar region $\Xz := \{ x \in \R \colon t^{-\frac{1}{\m}} |x| \lesssim t^{(\m-1) \rho} \}$, there exists a solution $Q =Q(y)$ to the nonlinear ordinary differential equation
\begin{equation} \label{eq:self}
\left| \frac{d}{dy} \right|^{\m-1} Q - y Q - \m Q^{\m} =0,
\end{equation}
satisfying $\| Q \|_{L^{\infty}_y} \lesssim \eps$ and
\begin{align*}
\left\| u(t) - t^{-\frac{1}{\m}} Q(t^{-\frac{1}{\m}} x) \right\|_{L^{\infty} (\Xz)} &\lesssim \eps t^{-\frac{1}{\m}-(\m-\frac{3}{2})\rho}, \\
\left\| u(t) - t^{-\frac{1}{\m}} Q(t^{-\frac{1}{\m}} x) \right\|_{L^2 (\Xz)} &\lesssim \eps t^{-\frac{1}{2\m}-(\m-1)\rho}.
\end{align*}
In the oscillatory region $\Xp := \{ x \in \R_+ \colon t^{-\frac{1}{\m}} |x| \gtrsim t^{(\m-1)\rho} \}$, there exists a unique complex-valued function $W$ satisfying $W(\xi)=\ol{W(-\xi)}$ and $\| W \|_{L^{\infty} \cap L^2} \lesssim \eps$ such that
\[
u(t) = \frac{2}{\sqrt{\m-1}} t^{-\frac{1}{\m}} (t^{-\frac{1}{\m}} x)^{-\frac{\m-2}{2(\m-1)}} \Re \left\{ W \left( t^{-\frac{1}{\m-1}} |x|^{\frac{1}{\m-1}} \right) e^{i \phi (t,x)} \right\} + \err_x,
\]
where the error, $\err_x$, satisfied the estimates
\[
\left\| t^{\frac{1}{\m}} (t^{-\frac{1}{\m}} x)^{\frac{3(\m-2)}{4(\m-1)}} \err_x \right\|_{L^{\infty}(\Xp)} + \left\| t^{\frac{1}{2\m}} ( t^{-\frac{1}{\m}} x)^{\frac{\m-2}{2(\m-1)}} \err_x \right\|_{L^2 (\Xp)} \lesssim \eps.
\]
In the corresponding frequency region $\Xf := \{ \xi \in \R_+ \colon t^{\frac{1}{\m}} |\xi| \gtrsim t^{\frac{1}{\m} (\frac{1}{2\m}-\eps)} \}$, we have
\[
\wh{u} (t,\xi) = W(\xi) e^{\frac{1}{\m} it \xi^{\m}} + \err_{\xi},
\]
where the error, $\err_{\xi}$, satisfies
\[
\left\| (t^{\frac{1}{\m}} \xi)^{\frac{\m-2}{4}} \err_{\xi} \right\|_{L^{\infty} (\Xf)} + \left\| t^{\frac{1}{2\m}} (t^{\frac{1}{\m}} \xi)^{\frac{\m-2}{2}} \err_{\xi} \right\|_{L^2 (\Xf)} \lesssim \eps.
\]
\end{thm}

As \eqref{hKdV} has time reversal symmetry given by $u(t,x)  \mapsto u(-t,-x)$, we get the corresponding asymptotic behavior as $t \rightarrow -\infty$.

Theorem \ref{thm} presents the leading asymptotic term not only in $L^{\infty}(\R)$, but also in $L^2(\R)$.
In addition, as with linear solutions, we divide the long-time behavior of the solution to \eqref{hKdV} into three distinct regions.
Moreover, Theorem \eqref{thm} says that there is a difference between $u$ and linear solutions in $\Xz$, while the leading term of $u$ in $\Xp$ behaves like a linear solution.

The regularity assumption $u_0 \in H^{\frac{2\m}{\m-1}}(\R)$ needs to show the existence of a local-in-time solution $u$ with $\lp (-t) u \in C([-T,T]; \wS^0(\R))$ (see Remark \ref{rmk:regularity} below).
In other words, regularity is no longer required in the proof of the global existence and asymptotic behavior.

\begin{rmk}
Because $Q$ satisfies \eqref{eq:self}, $\self (t,x) := t^{-\frac{1}{\m}} Q(t^{-\frac{1}{\m}} x)$ is a solution to \eqref{hKdV} with the initial datum $\self (0) = \int_{\R} u_0 (x) dx \delta_0$, where $\delta_0$ denotes the Dirac delta measure concentrated at the origin.
Moreover, by \eqref{selfest} below, we can roughly state that
\[
\| u(t) - \self (t) \|_{L^{\infty}(\R)} \lesssim \eps t^{-\frac{1}{\m}-\eps}.
\]
If $\int_{\R} u_0(x)dx=0$, then the self-similar solution vanishes, and the solution $u$ to \eqref{hKdV} decays faster than $t^{-\frac{1}{\m}}$.
Accordingly, the nonlinearity of \eqref{hKdV} is not critical in the long-time behavior in this case.
\end{rmk}

\begin{rmk} \label{rmk:shp}
Theorem \eqref{thm} is also true for short-range perturbations of the form
\begin{equation} \label{hKdVp}
\dt u + \frac{1}{\m} |\dx|^{\m-1} \dx u =  \dx (u^{\m} + F(u)),
\end{equation}
where there exists a real number $\p>\m$ such that $F \in C^3(\R)$ and
\[
\left| \frac{d^j}{du^j} F(u) \right| \lesssim |u|^{\p-j}
\]
for $j=0,1,2,3$.
In fact, if $u_0 \in \wS^{\frac{\p-1}{2\p}}(\R)$ with $\| u_0 \|_{\wS^{\frac{\p-1}{2\p}}} \le \eps$, then there exists a unique global solution $u$ to \eqref{hKdVp} with $\lp (-t) u \in C \big( \R; \wS^{\frac{\p-1}{2\p}}(\R) \big)$ satisfying the estimate \eqref{est:u_infty}.
Moreover, the global solution has the same asymptotic behavior as in Theorem \ref{thm}, as long as $\rho = \frac{1}{\m} (\frac{1}{2\m}-\eps)$ is replaced with $\wt{\rho} := \frac{1}{\m}(\frac{1}{2\m}-\dea-\eps)$, where $\dea :=\max (\frac{2\m+1-2\p}{2\m},0) <\frac{1}{2\m}$.
For completeness, we briefly outline these modifications in Appendix \ref{A:hKdVp}.
\end{rmk}

\subsection{Outline of the proof} \label{outliine}

We give an outline of the proof of Theorem \ref{thm}.
Let $\Lc$ denote the linear operator of \eqref{hKdV}:
\[
\Lc := \dt + \frac{1}{\m} |\dx|^{\m-1} \dx.
\]
To obtain pointwise estimates of solutions, we use the ``vector field''
\[
\J := \lp (t) x \lp (-t)
= x-t |\dx|^{\m-1}.
\]
However, $\J$ does not behave well with respect to the nonlinearity, so as in \cite{HayNau98, HayNau99, HarG16} we instead work with
\[
\Lambda := \dx^{-1} (\m t \dt + x \dx + 1).
\]
Here, $\Sc := \m t \dt + x \dx + 1$ is the generator of the scaling transformation for \eqref{hKdV}:
\begin{equation} \label{scaling}
u (t,x) \mapsto \lambda u(\lambda^{\m} t, \lambda x)
\end{equation}
for any $\lambda >0$.
Moreover, $\Sc$ is related to $\Lc$ and $\J$ as follows:
\begin{equation} \label{eq:SLJ}
\Sc = \m t \Lc + \J \dx + 1.
\end{equation}
We introduce the norm with respect to the spatial variable as follows:
\[
\| u (t) \|_{X} := \left( \| u (t) \|_{H^{\frac{\m-1}{2\m}}}^2 + \| \Lambda u (t) \|_{L^2}^2 \right)^{\frac{1}{2}}.
\]
We note that
\[
\| u_0 \|_{X} \sim \| u_0 \|_{H^{\frac{\m-1}{2\m}}} + \| x u _0 \|_{L^2} \sim  \| u_0 \|_{\wS^{\frac{\m-1}{2\m}}}.
\]

Well-posedness in $X$ follows from a similar argument as in \cite{KPV93}.

\begin{prop} \label{prop:WP1}
Let $\m \ge 4$.
If $u_0 \in \wS^{\frac{\m-1}{2\m}} (\R)$, then there exists a time $T=T \big( \| u_0 \|_{\wS^{\frac{\m-1}{2\m}}} \big) >0$ and a unique solution $u \in C([-T,T]; X)$ to \eqref{hKdV} satisfying
\[
\sup _{t \in [-T,T]} \| u(t) \|_{X} \lesssim \| u_0 \|_{\wS^{\frac{\m-1}{2\m}}}.
\]
Moreover, the flow map $u_0 \in \wS^{\frac{\m-1}{2\m}} (\R) \mapsto u \in C([-T,T]; X)$ is locally Lipschitz continuous.
\end{prop}

We give a proof of this well-posedness result in Appendix \ref{wp}.

\begin{rmk} \label{rmk:regularity}
Because \eqref{eq:SLJ} implies
\begin{equation} \label{eq:Ju}
\J u
= \Lambda u - \m t \dx^{-1} \Lc u
= \Lambda u - \m t u^{\m},
\end{equation}
the regularity $u_0 \in H^{\frac{\m-1}{2\m}} (\R)$ ensures $\J u \in C \left( [-T,T]; L^2(\R) \right)$, which is equivalent to $\lp (-t) u \in C \left( [-T,T]; \wS^0(\R) \right)$.
\end{rmk}

For initial data $u_0$ with $\| u_0 \|_{\wS^{\frac{\m-1}{2\m}}} \ll 1$, we can find an existence time $T>1$ and a unique solution $u \in C([-T,T]; X)$ to \eqref{hKdV}.
Moreover, Proposition \ref{prop:WP1} says that the existence of a global solution $u \in C(\R;X)$ follows from the decay estimate \eqref{est:u_infty}.

Because  \eqref{hKdV} is time reversal invariant, it suffices to consider the case $t \ge 0$.
We then make the bootstrap assumption that $u$ satisfies the linear pointwise estimates:
there exists a constant $D$ with $1 \ll D \ll \eps^{-\frac{1}{2}}$ such that
\begin{equation} \label{est:u_infty0}
\left\| \lr{t^{-\frac{1}{\m}}x}^{-\frac{k}{\m-1}+\frac{\m-2}{2(\m-1)}} \dx^k u(t) \right\|_{L^{\infty}} \le D \eps t^{-\frac{k+1}{\m}}
\end{equation}
for $t \in [1,T]$ and $k=0,1,\dots, \m-2$.

In \S \ref{S:energy}, under this assumption, for $\eps>0$ sufficiently small, we have the energy estimate
\begin{equation*}
\sup_{0 \le t \le T} \| u(t) \|_{X} \le C \eps \lr{T}^{\eps},
\end{equation*}
where $C$ is a constant independent of $D$ and $T$.
To complete the proof of global existence, we need to close the bootstrap estimate \eqref{est:u_infty0}.

In \S \ref{S:KS_type}, we prove decay estimates in $L^{\infty}(\R)$ and $L^2(\R)$ that allow us to reduce closing the bootstrap argument to considering the behavior of $u$ along the ray $\Gamma_v$.
We also observe that \eqref{est:u_infty0} holds true at $t=1$.

To close the bootstrap argument, we use the method of testing by wave packets as in \cite{HarG16, HIT, IfrTat15, Okapre}.
Here, a wave packet is an approximate solution localized in both space and frequency on the scale of the uncertainty principle.
Our main task in \S \ref{S:wave_packet} is to construct a wave packet  $\Psi_v(t,x)$ to the corresponding linear equation and observe its properties.
Here, to show that $\Psi_v(t,x)$ is a good approximate solution to the linear solution, we use the fact that $\m$ is a natural number.
Because $\Psi_v(t,x)$ is essentially frequency localized near $\xi_v = v^{\frac{1}{\m-1}}$ (see Lemma \ref{lem:freq_psi}), the linear operator $\Lc$ acts on $\Psi_v$ as $\dt + i(-i\dx)^{\m}$.
Hence, we can avoid applying the nonlocal operator $|\dx|$ directly in the calculation of $\Lc \Psi_v$ (see \eqref{eq:LPsi}).

To observe decay of $u$ along the ray $\Gamma_v$, we use the function
\begin{equation} \label{output}
\gamma (t,v) := \int_{\R} u(t,x) \ol{\Psi_v (t,x)} dx.
\end{equation}
In \S \ref{S:wave_packet}, we also prove that $\gamma$ is a reasonable approximation of $u$.
We then reduce closing the bootstrap estimate \eqref{est:u_infty0} to proving global bounds for $\gamma$.

In \S \ref{S:proof}, we show that $\gamma$ is the leading asymptotic term of $u$ in $\Xp$.
Moreover, we prove existence of a solution $Q$ to \eqref{eq:self}.

\subsection{Notation}
We summarize the notation used throughout this paper.
We set $\N_0 := \N \cup \{ 0 \}$.
We denote the space of all rapidly decaying functions on $\R$ by $\mathcal{S}(\R)$.
We define the Fourier transform of $f$ by $\mathcal{F}[f]$ or $\wh{f}$.

In estimates, we use $C$ to denote a positive constant that can change from line to line.
If $C$ is absolute or depends only on parameters that are considered fixed, then we often write $X \lesssim Y$, which means $X \le CY$.
When an implicit constant depends on a parameter $a$, we sometimes write $X \lesssim_{a} Y$.
We define $X \ll Y$ to mean $X \le C^{-1} Y$ and $X \sim Y$ to mean $C^{-1} Y \le X \le C Y$.
We write $X = Y + O(Z)$ when $|X-Y| \lesssim Z$.

Let $\sigma$ be a smooth even function with $0 \le \sigma \le 1$ and $\sigma (\xi ) = \begin{cases} 1, & \text{if } |\xi | \le 1, \\ 0, & \text{if } |\xi| \ge 2.\end{cases}$
For any $R, \, R_1, R_2 >0$ with $R_1 < R_2$, we set
\begin{gather*}
\sigma _{\le R} (\xi) := \sigma \left( \frac{\xi}{R} \right), \quad
\sigma _{>R} (\xi) := 1- \sigma_{\le R} (\xi), \\
\sigma _{<R}(\xi) := \sigma_{\le \frac{R}{2}}(\xi), \quad
\sigma _{\ge R}(\xi) := 1-\sigma_{<R}(\xi), \quad
\sigma _R (\xi) := \sigma_{\le R}(\xi) - \sigma_{<R}(\xi), \\
\sigma_{R_1 \le \cdot \le R_2} (\xi) := \sigma _{\le R_2}(\xi) - \sigma_{<R_1}(\xi), \quad
\sigma_{R_1 < \cdot < R_2} (\xi) := \sigma _{<R_2}(\xi) - \sigma_{\le R_1}(\xi).
\end{gather*}
For any $N, \, N_1, N_2 \in 2^{\mathbb{Z}}$ with $N_1 < N_2$, we define
\[
P_N f := \mathcal{F}^{-1} [\sigma_{N} \wh{f}] , \quad
P_{N_1 \le \cdot \le N_2} f := \mathcal{F}^{-1}[\sigma_{N_1 \le \cdot \le N_2} \wh{f}] .
\]
We denote the characteristic function of an interval $I$ by $\bm{1}_{I}$.
For $N \in 2^{\mathbb{Z}}$, we define
\[
P^{\pm} f := \mathcal{F}^{-1} [ \bm{1}_{\R_{\pm}} \wh{f}], \quad
P_N^{\pm} := P^{\pm} P_N.
\]


\section{Energy estimates} \label{S:energy}

We show energy estimates for solutions $u$ to \eqref{hKdV} under \eqref{est:u_infty0}.

\begin{lem} \label{lem:energy}
Assume $\m \ge 4$.
Let $u$ be a solution to \eqref{hKdV} in a time interval $[0,T]$ satisfying
\[
\| u_0 \|_{\wS^{\frac{\m-1}{2\m}}} \le \eps \ll 1
\]
and \eqref{est:u_infty0}.
Then,
\[
\| u(t) \|_{X} \lesssim \eps \lr{t}^{\eps},
\]
where the implicit constant is independent of $D$, $T$, and $\eps$.
\end{lem}

To treat the fractional derivatives, we use the Kato-Ponce commutator estimate (see \cite{KatPon88, KPV93}):

\begin{lem} \label{lem:KP}
For $0<s<1$, we have
\[
\| |\dx|^s (fg) - f |\dx|^s g \|_{L^2}
\lesssim \| |\dx|^s f \|_{L^2} \| g \|_{L^{\infty}}.
\]
\end{lem}

\begin{proof}[Proof of Lemma \ref{lem:energy}]
Because the desired bound for $0 \le t \le 1$ follows from Proposition \ref{prop:WP1}, we consider the case $t \ge 1$.

Integration by parts yields
\begin{align*}
\frac{1}{2} \dt \| u (t) \|_{L^2}^2
= \int_{\R} \dx (u^{\m}) u dx =0.
\end{align*}
By Lemma \ref{lem:KP} and \eqref{est:u_infty0}, we have
\begin{align*}
&\frac{1}{2} \dt \left\| |\dx|^{\frac{\m-1}{2\m}} u(t) \right\|_{L^2}^2 \\
&= \m \int_{\R} |\dx|^{\frac{\m-1}{2\m}} u \cdot |\dx|^{\frac{\m-1}{2\m}} (u^{\m-1} \dx u) dx \\
&= \m \int_{\R} |\dx|^{\frac{\m-1}{2\m}} u \cdot u^{\m-1} |\dx|^{\frac{\m-1}{2\m}} \dx u dx + O \left( \| u (t) \|_{H^{\frac{\m-1}{2\m}}}^2 \| u(t) \|_{L^{\infty}}^{\m-2} \| \dx u (t) \|_{L^{\infty}} \right) \\
&= - \frac{\m (\m-1)}{2} \int_{\R} u^{\m-2} \dx u \left( |\dx|^{\frac{\m-1}{2\m}} u \right)^2 dx + O \left( (D\eps)^{\m-1} t^{-1} \| u(t) \|_{H^{\frac{\m-1}{2\m}}}^2 \right) \\
& \lesssim (D\eps)^{\m-1} t^{-1} \| u(t) \|_{H^{\frac{\m-1}{2\m}}}^2.
\end{align*}
From $[ |\dx|^{\m-1} \dx, x] = \m |\dx|^{\m-1}$, a simple calculation yields
\[
[\Lc,\Sc ]= \m \Lc, \quad
[\Sc,\dx]=-\dx,
\]
which imply that for solutions $u$ to \eqref{hKdV}
\begin{equation} \label{eq:lambda}
\Lc \Lambda u
= \dx^{-1} (\Sc+\m) \Lc u
= \m u^{\m-1} \dx \Lambda u.
\end{equation}
From \eqref{eq:lambda}, integration by parts, and \eqref{est:u_infty0}, we obtain
\[
\frac{1}{2} \dt \| \Lambda u(t) \|_{L^2}^2
= - \frac{\m (\m-1)}{2} \int_{\R} u^{\m-2} \dx u (\Lambda u)^2 dx
\lesssim (D\eps)^{\m-1} t^{-1} \| \Lambda u(t) \|_{L^2}^2.
\]
From $(D\eps)^{\m-1} \ll \eps$, Gronwall's inequality with the above estimates implies
\[
\| u(t) \|_X
\le 10 \| u(1) \|_X t^{\eps}
\lesssim \eps t^{\eps}. \qedhere
\]
\end{proof}

We define the auxiliary space
\[
\| u(t) \|_{\wt{X}} := \| \J u(t) \|_{L^2} + t^{\frac{1}{\m}} \left\| \lr{t^{\frac{1}{\m}} \dx}^{-1} u(t) \right\|_{L^2}.
\]

\begin{lem} \label{lem:Xtilde}
Let $u$ be a solution to \eqref{hKdV} which satisfies $\| u_0 \|_{\wS^{\frac{\m-1}{2\m}}} \le \eps \ll 1$ and \eqref{est:u_infty0}.
Then, for $t \ge 1$, we have
\[
\| u(t) \|_{\wt{X}} \lesssim \eps t^{\frac{1}{2\m}},
\]
where the implicit constant is independent of $D$, $T$, and $\eps$
\end{lem}

\begin{proof}
By \eqref{est:u_infty0}, we see that
\begin{align*}
\left\| u^{\m} \right\|_{L^2}
&\le (D \eps)^{\m} t^{-1} \left( \int_{|x| \le t^{\frac{1}{\m}}} dx + \int_{|x| \ge t^{\frac{1}{\m}}} \left( t^{-\frac{1}{\m}} |x| \right)^{-\m \frac{\m-2}{\m-1}} dx \right)^{\frac{1}{2}} \\
&\lesssim \eps t^{-1+\frac{1}{2\m}}.
\end{align*}
From \eqref{eq:Ju} and Lemma \ref{lem:energy}, we therefore have
\[
\| \J u(t) \|_{L^2}
\lesssim \| \Lambda u(t) \|_{L^2} + t \left\| u^{\m} \right\|_{L^2}
\lesssim \eps t^{\frac{1}{2\m}}.
\]

We use a self-similar change of variables by defining
\begin{equation} \label{selfsimilar}
U(t,y) := t^{\frac{1}{\m}} u(t,t^{\frac{1}{\m}}y).
\end{equation}
A direct calculation shows
\begin{equation} \label{eq:U}
\dt U(t,y)
= \frac{1}{\m} t^{-1+\frac{1}{\m}} (\Sc u)(t, t^{\frac{1}{\m}}y)
= \frac{1}{\m} t^{-1} \dy \left( (\Lambda u) (t,t^{\frac{1}{\m}}y) \right).
\end{equation}
Hence, we have
\begin{equation} \label{eq:selfbound}
\dt \left\| \lr{\dy}^{-1} U(t) \right\|_{L^2_y}
\lesssim t^{-1-\frac{1}{2\m}} \| \Lambda u (t) \|_{L^2_x}
\lesssim \eps t^{-1-\frac{1}{2\m}+\eps} .
\end{equation}
By integrating this with respect to $t$, we have
\[
\left\| \lr{\dy}^{-1} U (t) \right\|_{L^2_y} \lesssim \eps .
\]
From $\left\| \lr{\dy}^{-1} U (t) \right\|_{L^2_y} = t^{\frac{1}{2\m}} \left\| \lr{t^{\frac{1}{\m}} \dx}^{-1} u (t) \right\|_{L^2_x}$, we obtain the desired bound.
\end{proof}

\begin{rmk} \label{rmk:reqH2}
The estimate $\| u(t) \|_{\wt{X}} \lesssim \eps$ for $0<t<1$ holds true.
Indeed, by Proposition \ref{prop:WP1}, Remark \ref{rmk:regularity}, and a sufficiently small value of $\eps>0$, we have
\begin{align*}
\sup_{0 \le t \le 1} \| u(t) \|_{\wt{X}}
&\lesssim \sup_{0 \le t \le 1} \left( \| \Lambda u (t) \|_{L^2} + \| u(t)^{\m} \|_{L^2} + \| u (t) \|_{L^2} \right) \\
&\lesssim \sup_{0 \le t \le 1} \left( \| u(t) \|_X + \| u(t) \|_X^{\m} \right) \\
&\lesssim \eps.
\end{align*}
\end{rmk}

\section{Decay estimates} \label{S:KS_type}

In this section, we prove a number of estimates for $u$ without assuming \eqref{est:u_infty0}.

We divide $u$ into two parts on which $\J$ acts hyperbolically and elliptically.
For simplicity, we use the following notation:
\[
u_N := P_N u, \quad u^{\pm} := P^{\pm} u, \quad u_N^+ := P_N^+ u.
\]
Since $\J u_N = P_N (\J u) +i \mathcal{F}^{-1} [ \partial_{\xi} \sigma_N \wh{u}]$ and $u$ is real valued, we have
\begin{equation} \label{freqXt}
\| u (t) \| _{\wt{X}} \sim \Bigg( \Big\| \uel (t) \Big\|_{\wt{X}}^2 + \sum_{\substack{N \in 2^{\mathbb{Z}} \\ N > t^{-\frac{1}{\m}}}} \left\| u_N^+ (t) \right\| _{\wt{X}}^2 \Bigg)^{\frac{1}{2}},
\quad
\uel := \sum _{\substack{N \in 2^{\mathbb{Z}} \\ N \le t^{-\frac{1}{\m}}}} u_N.
\end{equation}
For $t \ge 1$ and $N > t^{-\frac{1}{\m}}$, we define the hyperbolic and elliptic parts of $u_N^+$ as follows:
\[
\uhp_{N} := \sigma_N^{\rm{hyp}} u^{+}_{N}, \quad
\uep_{N} := u^{+}_{N} - \uhp_{N},
\]
where $\sigma_N^{\rm{hyp}} (t,x) := \sigma_{\frac{1}{\kk} t N^{\m-1} \le \cdot \le \kk t N^{\m-1}} (x) \bm{1}_{\R_{+}} (x)$ and $\kk := 2^{2\m+3}$.
This large constant $\kk$ is needed to show \eqref{uepNl2} in Lemma \ref{lem:he_freq_est} below.
Moreover, we define
\[
\uhp = \sum _{\substack{N \in 2^{\mathbb{Z}} \\ N > t^{-\frac{1}{\m}}}} \uhp_{N}, \quad
\ue = u - 2\Re \uhp.
\]
We note that $\uhp$ is supported in $\{ x \in \R_{+} \colon t^{-\frac{1}{\m}} x\ge \frac{1}{2 \kk} \}$.
For $(t,x) \in \R^2$ with $t^{-\frac{1}{\m}} |x| \ge \frac{1}{2 \kk}$, the number of dyadic numbers $N \in 2^{\mathbb{Z}}$ satisfying $\frac{1}{2 \kk} tN^{\m-1} \le |x| \le 2 \kk t N^{\m-1}$ is less than $10$.
Hence, $\uhp (t,x)$ is a finite sum of $\uhp_N(t,x)$'s.

The functions $\uhp_N$ and $\uep_N$ are essentially localized at frequency $N$ in the following sense:
For any $a \ge 0$, $b \in \R$, and $c \ge 0$,
\begin{align}
\label{est:uhperr}
\left\| \left( 1-\PN^{+} \right) |\dx|^{a} \left( |x|^{b} \uhp_N \right) \right\|_{L^2}
&\lesssim_{a,b,c} t^{-\frac{a-b}{\m}} \left( t^{\frac{1}{\m}} N \right)^{-c} \| u_N \|_{L^2},
\\
\label{est:ueperr}
\left\| \left( 1-\PN^{+} \right) |\dx|^{a} \uep_N \right\|_{L^2}
&\lesssim_{a,c,n} t^{-\frac{a}{\m}} \left( t^{\frac{1}{\m}} N \right)^{-c} \| u_N \|_{L^2},
\\
\label{est:ueperr2}
\left\| \left( 1-\PN^{+} \right) |\dx|^{a} \left( |x|^b \sigma_{>t^{\frac{1}{\m}}} (x) \uep_N \right) \right\|_{L^2}
&\lesssim_{a,b,c} t^{-\frac{a-b}{\m}} \left( t^{\frac{1}{\m}} N \right)^{-c} \| u_N \|_{L^2}.
\end{align}
These estimates are consequences of the following lemma (see, for example, \cite{Oka17, Okapre}):
\begin{lem} \label{lem:freq_spat_loc}
For $2 \le p \le \infty$, any $a ,\, b ,\, c \in \R$ with $a \ge 0$ and $a+c \ge 0$, and any $R>0$, we have
\[
\left\| \left( 1- \PN^+ \right) | \dx| ^{a} ( |x|^{b} \sigma_R P_N^+ f) \right\|_{L^p}
\lesssim_{a,b,c} N^{-c+\frac{1}{2}-\frac{1}{p}} R^{-a+b-c}  \| P_N^+ f \|_{L^2} . 
\]
Moreover, we may replace $\sigma_R$ on the left hand side by $\sigma_{>R}$ if $a+c>b+1$ and by $\sigma_{<R}$ if $a+c \ge 0$ and $b=0$.
\end{lem}

\subsection{Frequency localized estimates}

Since
\[
\J u^+  =
\left( x-t (-i\dx)^{\m-1} \right) u^+,
\]
by factorizing out the term $x- t \xi^{\m-1}$, we define
\[
\J_{+} := |x|^{\frac{1}{\m-1}} + it^{\frac{1}{\m-1}} \dx, \quad
\J_{-} := \sum_{j=0}^{\m-2} |x|^{\frac{\m-2-j}{\m-1}} \left( -i t^{\frac{1}{\m-1}} \dx \right)^j.
\]
These operators are useful in our analysis.

We begin with the following preliminary observation.

\begin{lem} \label{lem:eqs}
Let $a$ be a real number and let $g$ be a ($\C$-valued) smooth function supported in $\R_+$ or $\R_-$.
For any integer $k>0$, the following equations hold:
\begin{align*}
\Re \int_{\R_{\pm}} |x|^a g(x) \ol{\dx^{2k-1} g(x)} dx
&= \sum_{l=1}^{k} C_{2k-1,l}^{a,\pm} \int_{\R_{\pm}} \left| |x|^{\frac{a+1}{2}-l} \dx^{k-l} g(x) \right|^2 dx, \\
\Re \int_{\R_{\pm}} |x|^a g(x) \ol{\dx^{2k} g(x)} dx
&= \sum_{l=0}^{k} C_{2k,l}^{a,\pm} \int_{\R_{\pm}} \left| |x|^{\frac{a}{2}-l} \dx^{k-l} g(x) \right|^2 dx, \\
\Im \int_{\R_{\pm}} |x|^a g(x) \ol{\dx^{2k-1} g(x)} dx
&= \sum_{l=0}^{k-1} D_{2k-1,l}^{a,\pm} \int_{\R} \xi \left| \mathcal{F} \left[ | \cdot |^{\frac{a}{2}-l} \dx^{k-1-l} g \right] (\xi) \right|^2 d\xi, \\
\Im \int_{\R_{\pm}} |x|^a g(x) \ol{\dx^{2k} g(x)} dx
&= \sum_{l=1}^{k} D_{2k,l}^{a,\pm} \int_{\R} \xi \left| \mathcal{F} \left[ | \cdot |^{\frac{a+1}{2}-l} \dx^{k-l} g \right] (\xi) \right|^2 d\xi,
\end{align*}
where $C_{k,l}^{a,\pm}$ and $D_{k,l}^{a,\pm}$ are real constants depending on $a$, $k$, and $l$.
In particular, $C_{2k,0}^{a,\pm} = D_{2k-1,0}^{a,\pm} = (-1)^k$.
\end{lem}

\begin{proof}
For $k=1$, integration by parts yields
\begin{align*}
&
\begin{aligned}
\Re \int_{\R_{\pm}} |x|^a g(x) \ol{\dx g(x)} dx
= \mp \frac{a}{2} \int_{\R_{\pm}} \left| |x|^{\frac{a-1}{2}} g(x) \right|^2 dx,
\end{aligned}
\\
&
\begin{aligned}
& \Re \int_{\R_{\pm}} |x|^a g(x) \ol{\dx^{2} g(x)} dx \\
& = -\Re \int_{\R_{\pm}} \left( |x|^a \dx g(x) \pm a |x|^{a-1} g(x)  \right) \ol{\dx g (x)} dx \\
& = - \int_{\R_{\pm}} \left| |x|^{\frac{a}{2}} |\dx g(x) \right|^2dx + \frac{1}{2} a(a-1) \int_{\R_{\pm}} \left| |x|^{\frac{a}{2}-1} g(x) \right|^2 dx.
\end{aligned}
\end{align*}
Similarly, we have
\begin{align*}
&
\begin{aligned}
\Im \int_{\R_{\pm}} |x|^a g(x) \ol{\dx g(x)} dx
&= \Im \int_{\R_{\pm}} |x|^{\frac{a}{2}} g(x) \ol{\dx \left( |\cdot|^{\frac{a}{2}} g \right) (x)} dx \\
&= - \int_{\R} \xi \left| \mathcal{F} \left[ |\cdot |^{\frac{a}{2}} g \right] (\xi) \right|^2 d\xi ,
\end{aligned}
\\ &
\begin{aligned}
\Im \int_{\R_{\pm}} |x|^a g(x) \ol{\dx^2 g(x)} dx
& = - \Im \int_{\R_{\pm}} \left( |x|^a \dx g(x) \pm a|x|^{a-1} g(x) \right) \ol{\dx g(x)} dx \\
& = \pm a \int_{\R} \xi \left| \mathcal{F} \left[ |\cdot |^{\frac{a-1}{2}} g \right] (\xi) \right|^2 d\xi .
\end{aligned}
\end{align*}
Hence, the equations hold when $k=1$ with $C_{1,1}^{a,\pm} =\mp \frac{a}{2}$, $C_{2,0}^{a,\pm}=-1$, $C_{2,1}^{a,\pm}=\frac{1}{2}a(a-1)$, $D_{1,0}^{a,\pm}=-1$, and $D_{2,1}^{a,\pm}=\pm a$.

Next, we assume that these equalities  hold up to $k-1$.
Integration by parts yields
\begin{align*}
& \begin{aligned}
& \Re \int_{\R_{\pm}} |x|^a g(x) \ol{\dx^{2k-1} g(x)} dx \\
& = - \Re \int_{\R_{\pm}} |x|^a \dx g(x) \ol{\dx^{2(k-1)-1} \dx g(x)} dx \mp a \Re \int_{\R_{\pm}} |x|^{a-1} g(x) \ol{\dx^{2(k-1)} g(x)} dx \\
& = - \sum_{l=1}^{k-1} C_{2k-3,l}^{a,\pm} \int_{\R_{\pm}} \left| |x|^{\frac{a+1}{2}-l} \dx^{k-l} g(x) \right|^2 dx \\
& \quad \mp a \sum_{l=0}^{k-1} C_{2k-2,l}^{a-1,\pm} \int_{\R_{\pm}} \left|  |x|^{\frac{a-1}{2}-l} \dx^{k-1-l} g] (x) \right|^2 dx \\
& = \sum_{l=1}^{k-1} \left( -C_{2k-3,l}^{a,\pm} \mp a C_{2k-2,l-1}^{a-1,\pm} \right) \int_{\R_{\pm}} \left| |x|^{\frac{a+1}{2}-l} \dx^{k-l} g (x) \right|^2 dx \\
& \quad \mp a C_{2k-2,k-1}^{a-1,\pm} \int_{\R_{\pm}} \left| |x|^{\frac{a+1}{2}-k} g(x) \right|^2 dx,
\end{aligned}
\\ &
\begin{aligned}
& \Re \int_{\R_{\pm}} |x|^a g(x) \ol{\dx^{2k} g(x)} dx \\
& = - \Re \int_{\R_{\pm}} |x|^a \dx g(x) \ol{\dx^{2(k-1)} \dx g(x)} dx \mp a \Re \int_{\R_{\pm}} |x|^{a-1} g(x) \ol{\dx^{2k-1} g(x)} dx \\
& = - \sum_{l=0}^{k-1} C_{2k-2,l}^{a,\pm} \int_{\R_{\pm}} \left| |x|^{\frac{a}{2}-l} \dx^{k-l} g(x) \right|^2 dx \\
& \quad \mp a \sum_{l=1}^{k} C_{2k-1,l}^{a-1,\pm} \int_{\R_{\pm}} \left| |x|^{\frac{a}{2}-l} \dx^{k-l} g] (x) \right|^2 dx \\
& = -C_{2k-2,0}^{a,\pm} \int_{\R_{\pm}} \left| |x|^{\frac{a}{2}} \dx^{k} g(x) \right|^2 dx \mp a C_{2k-1,k}^{a-1,\pm} \int_{\R_{\pm}} \left| |x|^{\frac{a}{2}-k} g(x) \right|^2 dx \\
& \quad + \sum_{l=1}^{k-1} \left( -C_{2k-2,l}^{a,\pm} \mp a C_{2k-1,l}^{a-1,\pm} \right) \int_{\R_{\pm}} \left| |x|^{\frac{a}{2}-l} \dx^{k-l} g (x) \right|^2 dx .
\end{aligned}
\end{align*}
Hence, by setting
\begin{align*}
C_{2k-1,l}^{a,\pm} &:= \begin{cases} -C_{2k-3,l}^{a,\pm} \mp a C_{2k-2,l-1}^{a-1,\pm}, & \text{if $l=1,2, \dots, k-1$}, \\ \mp a C_{2k-2,k-1}^{a-1,\pm}, & \text{if $l=k$}, \end{cases} \\
C_{2k,l}^{a,\pm} &:= \begin{cases} -C_{2k-2,0}^{a,\pm}, & \text{if $l=0$}, \\ -C_{2k-2,l}^{a,\pm} \mp aC_{2k-1,l}^{a-1,\pm}, & \text{if $l=1,2, \dots , k-1$}, \\ \mp a C_{2k-1,k}^{a,\pm}, & \text{if $l=k$}, \end{cases}
\end{align*}
we obtain the equations for the real part.
Similarly, by setting
\begin{align*}
D_{2k-1,l}^{a,\pm} &:= \begin{cases} -D_{2k-3,0}^{a,\pm}, & \text{if $l=0$}, \\ -D_{2k-3,l}^{a,\pm} \mp aD_{2k-2,l}^{a-1,\pm}, & \text{if $l=1,2, \dots , k-2$}, \\ \mp a D_{2k-2,k-1}^{a,\pm}, & \text{if $l=k-1$}, \end{cases} \\
D_{2k,l}^{a,\pm} &:= \begin{cases} -D_{2k-2,l}^{a,\pm} \mp a D_{2k-1,l-1}^{a-1,\pm}, & \text{if $l=1,2, \dots, k-1$}, \\ \mp a D_{2k-1,k-1}^{a-1,\pm}, & \text{if $l=k$}, \end{cases}
\end{align*}
we obtain the equations for the imaginary part.
From the recurrence relations with $C_{2,0}^{a,\pm}=D_{1,0}^{a,\pm}=-1$, we have $C_{2k,0}^{a,\pm}=D_{2k-1,0}^{a,\pm}=(-1)^k$.
\end{proof}

We show the frequency localized estimates.

\begin{lem} \label{lem:he_freq_est}
For $t \ge 1$ and $N > t^{-\frac{1}{\m}}$, we have
\begin{align}
\left\| \left( |x|^{\frac{\m-2}{\m-1}} + t^{\frac{\m-2}{\m-1}} N^{\m-2} \right) \J_{+} \uhp_{N} (t) \right\|_{L^2} &\lesssim \| u_N (t) \|_{\wt{X}}, \label{uhpNl2} \\
\left\| \left( |x|+ t N^{\m-1} \right) \uep_{N} (t) \right\|_{L^2} &\lesssim \| u_N (t) \|_{\wt{X}}. \label{uepNl2}
\end{align}
\end{lem}

\begin{proof}
Set $f := \J_{+} \uhp_N$.
We apply Lemma \ref{lem:eqs} to obtain
\begin{align*}
& \| \J_{-} f \|_{L^2}^2 - \sum_{j=0}^{\m-2} \left\| t^{\frac{j}{\m-1}} |x|^{\frac{\m-2-j}{\m-1}} \dx^j f \right\|_{L^2}^2 \\
&= 2 \sum_{j=0}^{\m-3} \sum_{k=1}^{\m-2-j} \Re \int_{\R} i^{k} t^{\frac{2j+k}{\m-1}} |x|^{2-\frac{2j+k+2}{\m-1}} \dx^j f(x) \ol{\dx^{j+k} f(x)} dx \\
&= 2 \sum_{j=0}^{\m-3} \sum_{\substack{1 \le k \le \m-2-j \\ k : \text{ even}}} \Re \int_{\R} (-1)^{\frac{k}{2}} t^{\frac{2j+k}{\m-1}} |x|^{2-\frac{2j+k+2}{\m-1}} \dx^j f(x) \ol{\dx^{j+k} f(x)} dx \\
& \quad + 2 \sum_{j=0}^{\m-3} \sum_{\substack{1 \le k \le \m-2-j \\ k \text{: odd}}} \Im \int_{\R} (-1)^{\frac{k+1}{2}} t^{\frac{2j+k}{\m-1}} |x|^{2-\frac{2j+k+2}{\m-1}} \dx^j f(x) \ol{\dx^{j+k} f(x)} dx \\
&= 2 \sum_{j=0}^{\m-3} \sum_{\substack{1 \le k \le \m-2-j \\ k : \text{ even}}} \sum_{l=0}^{\frac{k}{2}} (-1)^{\frac{k}{2}} C_{k,l}^{2-\frac{2j+k+2}{\m-1},+} t^{\frac{2j+k}{\m-1}} \int_{\R} \left| |x|^{1-\frac{2j+k+2}{2(\m-1)}-l} \dx^{j+\frac{k}{2}-l} f(x) \right|^2 dx \\
& \quad + 2 \sum_{j=0}^{\m-3} \sum_{\substack{1 \le k \le \m-2-j \\ k \text{: odd}}} \sum_{l=0}^{\frac{k-1}{2}} (-1)^{\frac{k+1}{2}} D_{k,l}^{2-\frac{2j+k+2}{\m-1},+} t^{\frac{2j+k}{\m-1}} \\
& \hspace*{50pt} \times \int_{\R} \xi \left| \mathcal{F} \left[ |\cdot |^{1-\frac{2j+k+2}{2(\m-1)}-l} \dx^{j+\frac{k-1}{2}-l} f \right] (\xi) \right|^2 dx \\
&\gtrsim - \sum_{j=0}^{\m-3} \sum_{\substack{1 \le k \le \m-2-j \\ k : \text{ even}}} \sum_{l=1}^{\frac{k}{2}} t^{\frac{2j+k}{\m-1}} \int_{\R} \left| |x|^{1-\frac{2j+k+2}{2(\m-1)}-l} \dx^{j+\frac{k}{2}-l} f(x) \right|^2 dx \\
& \quad - \sum_{j=0}^{\m-3} \sum_{\substack{1 \le k \le \m-2-j \\ k \text{: odd}}} \Bigg( t^{\frac{2j+k}{\m-1}} \int_{\R_{-}} |\xi| \left| \mathcal{F} \left[ |\cdot |^{1-\frac{2j+k+2}{2(\m-1)}} \dx^{j+\frac{k-1}{2}} f \right] (\xi) \right|^2 dx \\
& \hspace*{100pt} + \sum_{l=1}^{\frac{2k-1}{2}} t^{\frac{2j+k}{\m-1}} \int_{\R} |\xi| \left| \mathcal{F} \left[ |\cdot |^{1-\frac{2j+k+2}{2(\m-1)}-l} \dx^{j+\frac{k-1}{2}-l} f \right] (\xi) \right|^2 dx \Bigg).
\end{align*}
Here, \eqref{est:uhperr} yields that for $j \ge 1$
\begin{align*}
&
\begin{aligned}
& t^{\frac{2j+k}{\m-1}} \int_{\R} \left| |x|^{1-\frac{2j+k+2}{2(\m-1)}-l} \dx^{j+\frac{k}{2}-l} f(x) \right|^2 dx \\
& \lesssim t^{\frac{2j+k}{\m-1}} (tN^{\m-1})^{2-\frac{2j+k+2}{\m-1}-2l} \| \dx^{j+\frac{k}{2}-l} \J_{+} \uhp_N \|_{L^2}^2 \\
& \lesssim (tN^{\m})^{2(1-l)-\frac{2}{\m-1}} N^{-2j-k+2l-2+\frac{2}{\m-1}} \bigg( N^{2j+k-2l} \left\| \PN^{+} \J_{+} \uhp_N \right\|_{L^2}^2 \\
& \qquad + \left\| \left( 1-\PN^{+} \right) \dx^{j+\frac{k}{2}-l} \J_{+} \uhp_N \right\|_{L^2}^2 \bigg) \\
& \lesssim (tN^{\m})^{2(1-l)} N^{-2} \| u_N \|_{L^2}^2 + N^{-2} \| u_N \|_{L^2}^2 \\
& \lesssim N^{-2} \| u_N \|_{L^2}^2, \\
\end{aligned}
\\
&
\begin{aligned}
& t^{\frac{2j+k}{\m-1}} \int_{\R} |\xi| \left| \mathcal{F} \left[ |\cdot |^{1-\frac{2j+k+2}{2(\m-1)}-l} \dx^{j+\frac{k-1}{2}-l} f \right] (\xi) \right|^2 dx \\
& \lesssim t^{\frac{2j+k}{\m-1}} N \left\| \PN^{+} \left( |x|^{1-\frac{2j+k+2}{2(\m-1)}-l} \dx^{j+\frac{k-1}{2}-l} \J_{+} \uhp_N \right) \right\|_{L^2}^2 \\
& \quad + t^{\frac{2j+k}{\m-1}} \left\| \left( 1-\PN^{+} \right) |\dx|^{\frac{1}{2}} \left( |x|^{1-\frac{2j+k+2}{2(\m-1)}-l} \dx^{j+\frac{k-1}{2}-l} \J_{+} \uhp_N \right) \right\|_{L^2}^2 \\
& \lesssim (tN^{\m})^{2(1-l)} N^{-2} \| u_N \|_{L^2}^2 + N^{-2} \| u_N \|_{L^2}^2 \\
& \lesssim N^{-2} \| u_N \|_{L^2}^2.
\end{aligned}
\end{align*}
Similarly, we get
\begin{align*}
& t^{\frac{2j+k}{\m-1}} \int_{\R_{-}} |\xi| \left| \mathcal{F} \left[ |\cdot |^{1-\frac{2j+k+2}{2(\m-1)}} \dx^{j+\frac{k-1}{2}} f \right] (\xi) \right|^2 dx \\
& \le t^{\frac{2j+k}{\m-1}} \left\| P^{-} |\dx|^{\frac{1}{2}} \left( |x|^{1-\frac{2j+k+2}{2(\m-1)}} \dx^{j+\frac{k-1}{2}} \J_{+} \uhp_N \right) \right\|_{L^2}^2 \\
& \lesssim N^{-2} \| u_N \|_{L^2}^2.
\end{align*}
Hence, we have
\begin{equation} \label{eq:uhpJ-}
\| \J_{-} f \|_{L^2}^2
\gtrsim \sum_{j=0}^{\m-2} \left\| t^{\frac{j}{\m-1}} |x|^{\frac{\m-2-j}{\m-1}} \dx^j f \right\|_{L^2}^2 - N^{-2} \| u_N \|_{L^2}^2.
\end{equation}

On the other hand, because
\begin{gather*}
\begin{aligned}
(\J_{-} \J_{+} \uhp_N)(x)
=& \left( x-t(- i \dx)^{\m-1} \right) \uhp_N(x) \\
& + \sum_{j=1}^{\m-2} (- i t^{\frac{1}{\m-1}})^j |x|^{\frac{\m-2-j}{\m-1}} \sum_{k=0}^{j-1} \begin{pmatrix} j \\ k \end{pmatrix} \dx^{j-k} \left( x^{\frac{1}{\m-1}} \right) \dx^k \uhp_N,
\end{aligned}
\\
\left( x-t(- i \dx)^{\m-1} \right) \uhpn_N
=\J \uhp_N - t \left( (-i\dx)^{\m-1}-|\dx|^{\m-1} \right) \uhp_N,
\end{gather*}
\eqref{est:uhperr} and $tN^{\m} > 1$ yield
\begin{align*}
\| \J_{-} f \|_{L^2}
& \lesssim \left\| \left( x-t(- i \dx)^{\m-1} \right) \uhp_N \right\|_{L^2} + \sum_{j=1}^{\m-2} \sum_{k=0}^{j-1} t^{\frac{j}{\m-1}} \left\| |x|^{1-\frac{\m j}{\m-1}+k} \dx^k \uhp_N \right\|_{L^2} \\
& \lesssim \| \J \uhp_N \|_{L^2} + \sum_{j=1}^{\m-2} \sum_{k=0}^{j-1} (t N^{\m})^{1-j+k} N^{-1} \| u_N \|_{L^2} + N^{-1} \| u_N \|_{L^2} \\
& \lesssim \|u_N (t) \|_{\wt{X}}.
\end{align*}
Combining this with \eqref{eq:uhpJ-}, we obtain \eqref{uhpNl2}.

For the elliptic bound, we decompose $\uep_N$ into three parts
\[
\uep_N = \sigma _{\le \frac{1}{\kk} tN^{\m-1}} \uep_{N} + \sigma _{\frac{1}{\kk} tN^{\m-1} < \cdot < \kk tN^{\m-1}} \uep_N + \sigma _{\ge \kk tN^{\m-1}} \uep_N.
\]
We observe that the equation
\begin{equation} \label{eq:ueob}
\left\| x f \right\|_{L^2}^2 + \left\| t |\dx|^{\m-1} f \right\|_{L^2}^2 = \left\| \J f \right\|_{L^2}^2 + 2 \int_{\R} tx \left| \dx^{\frac{\m-1}{2}} f(x) \right|^2 dx
\end{equation}
holds for any smooth function $f$ and odd $\m$.
Similarly, for even $\m$,
\begin{equation} \label{eq:ueobeven}
\left\| x |\dx|^{\frac{3}{2}} f \right\|_{L^2}^2 + \left\| t |\dx|^{\m+\frac{1}{2}} f \right\|_{L^2}^2 = \left\| \J |\dx|^{\frac{3}{2}} f \right\|_{L^2}^2 + 2 \int_{\R} tx \left| \dx^{\frac{\m}{2}+1} f(x) \right|^2 dx
\end{equation}
holds for any smooth function $f$.

In what follows, we only consider the case when $\m$ is even, because the case when $\m$ is odd can be similarly handled.

Since $[x, \dx^{\frac{\m}{2}+1} |\dx|^{-\frac{3}{2}}] = - (\frac{\m}{2}+1) \dx^{\frac{\m}{2}} |\dx|^{-\frac{3}{2}} - \frac{3}{2} \dx^{\frac{\m}{2}+2} |\dx|^{-\frac{7}{2}}$ and $\kk = 2^{2\m+3}$, \eqref{est:ueperr2} implies
\begin{align*}
& \left| \int_{\R} tx \left| \dx^{\frac{\m}{2}+1} \left( \sigma _{\ge \kk tN^{\m-1}} \uep_{N} \right) (t,x) \right|^2 dx \right| \\
& \le \frac{2}{\kk} N^{-\m+1} \left\| x \dx^{\frac{\m}{2}+1} \left( \sigma _{\ge \kk tN^{\m-1}} \uep_{N} (t) \right) \right\|_{L^2}^2 \\
& \le \frac{4}{\kk} N^{-\m+1} \left\| |\dx|^{\frac{\m-1}{2}} \left\{ x |\dx|^{\frac{3}{2}} \left( \sigma _{\ge \kk tN^{\m-1}} \uep_{N} (t) \right) \right\} \right\|_{L^2}^2 \\
& \quad + C N^{-\m+1} \left\| \dx^{\frac{\m}{2}} \left( \sigma _{\ge \kk tN^{\m-1}} \uep_{N} (t) \right) \right\|_{L^2}^2 \\
& \le \frac{2^{2\m+1}}{\kk} \left\| x |\dx|^{\frac{3}{2}} \left( \sigma _{\ge \kk tN^{\m-1}} \uep_{N} (t) \right) \right\|_{L^2}^2 \\
& \quad + C N^{-\m+1} \left\| \left( 1-\PN^{+} \right) |\dx|^{\frac{\m-1}{2}} \left( x |\dx|^{\frac{3}{2}}  \sigma _{\ge \kk tN^{\m-1}} \uep_{N} (t) \right) \right\|_{L^2}^2 \\
& \quad + C N \| u_N (t) \|_{L^2}^2 \\
& \le \frac{1}{4} \left\| x |\dx|^{\frac{3}{2}} \left( \sigma _{\ge \kk tN^{\m-1}} \uep_{N} (t) \right) \right\|_{L^2}^2 + C N \| u_{N} (t) \|_{L^2}^2.
\end{align*}
Taking $f= \sigma _{\ge \kk tN^{\m-1}} \uep_{N}$ in \eqref{eq:ueobeven}, we have
\begin{align*}
&\left\| x |\dx|^{\frac{3}{2}} \left( \sigma _{\ge \kk tN^{\m-1}} \uep_{N} (t) \right) \right\|_{L^2} \\
& \lesssim \left\| \J |\dx|^{\frac{3}{2}} \left( \sigma _{\ge \kk tN^{\m-1}} \uep_{N} (t) \right) \right\|_{L^2} + N^{\frac{1}{2}} \| u_N (t) \|_{L^2},
\end{align*}
which shows
\[
\left\| x \sigma _{\ge \kk tN^{\m-1}} \uep_{N} (t) \right\|_{L^2}
\lesssim \| u_N (t) \|_{\wt{X}}.
\]

By \eqref{est:ueperr}, we have
\begin{align*}
& \left| \int_{\R} tx \left| \dx^{\frac{\m}{2}+1} \left( \sigma _{\le \frac{1}{\kk} tN^{\m-1}} \uep_{N} \right) (t,x) \right|^2 dx \right| \\
& \le \frac{2}{\kk} N^{\m-1} \left\| t \dx^{\frac{\m}{2}+1} \left( \sigma _{\le \frac{1}{\kk} tN^{\m-1}} \uep_{N} (t) \right) \right\|_{L^2}^2 \\
& \le \frac{2^{2\m}}{\kk} \left\| \PN^{+} t |\dx|^{\m+\frac{1}{2}} \left( \sigma _{\le \frac{1}{\kk} tN^{\m-1}} \uep_{N} (t) \right) \right\|_{L^2}^2  \\
& \quad + C N^{\m-1} \left\| \left( 1-\PN^{+} \right) t \dx^{\frac{\m}{2}+1} \left( \sigma _{\le \frac{1}{\kk} tN^{\m-1}} \uep_{N} (t) \right) \right\|_{L^2}^2 \\
& \le \frac{1}{8} \left\| t |\dx|^{\m+\frac{1}{2}} \left( \sigma _{<\frac{1}{\kk} tN^{\m-1}} \uep_{N} (t) \right) \right\|_{L^2}^2 + C N \| u_N (t) \|_{L^2}^2 .
\end{align*}
Taking $f= \sigma _{\le \frac{1}{\kk} tN^{\m-1}} \uep_{N}$ in \eqref{eq:ueobeven}, we have
\begin{align*}
& \left\| t |\dx|^{\m+\frac{1}{2}} \left( \sigma _{\le \frac{1}{\kk} tN^{\m-1}} \uep_{N} (t) \right) \right\|_{L^2} \\
& \lesssim \left\| \J |\dx|^{\frac{3}{2}} \left( \sigma _{\le \frac{1}{\kk} tN^{\m-1}} \uep_{N} (t) \right) \right\|_{L^2} + N^{\frac{1}{2}} \| u_N (t) \|_{L^2},
\end{align*}
which shows that
\[
\left\| t |\dx|^{\m-1} \left( \sigma _{\le \frac{1}{\kk} tN^{\m-1}} \uep_{N} (t) \right) \right\|_{L^2}
\lesssim \| u_N (t) \|_{\wt{X}}.
\]

From
\[
\int_{\R} tx \left| \dx^{\frac{\m}{2}+1} \left( \sigma _{\frac{1}{\kk} tN^{\m-1} < \cdot < \kk tN^{\m-1}} (x) \uep_{N} (t,x) \right) \right|^2 dx <0,
\]
taking $f=\sigma _{\frac{1}{\kk} tN^{\m-1} < \cdot < \kk tN^{\m-1}} \uep_{N}$ in \eqref{eq:ueobeven}, we have
\[
tN^{\m-1} \left\| \sigma _{\frac{1}{\kk} tN^{\m-1} < \cdot < \kk tN^{\m-1}} \uep_{N} (t) \right\|_{L^2}
\lesssim \| u_N (t) \|_{\wt{X}}. \qedhere
\]
\end{proof}

\subsection{Decay estimates in $L^2(\R)$ and $L^{\infty}(\R)$}

First, by summing up the frequency localized estimates, we show the $L^2$-estimates.

\begin{cor} \label{cor:uL2}
For $t \ge 1$, we have
\begin{align}
& \sum_{k=0}^{\m-2} \sum_{l=0}^k \left\| t^{\frac{k+1}{\m-1}} |x|^{-\frac{\m k+1}{\m-1}+l} \dx^l \uhp \right\|_{L^2} \lesssim \| u(t) \|_{\wt{X}}, \label{uhpL22} \\
& \sum_{k=0}^{\m-2} \left\| t^{\frac{k}{\m-1}} |x|^{-\frac{k-\m+2}{\m-1}} \J_+ \dx^k \uhp \right\| _{L^2} \lesssim \| u (t) \|_{\wt{X}}, \label{uhpL2} \\
& \sum_{k=0}^{\m-2} \left\| t^{\frac{k+1}{\m}} \lr{t^{-\frac{1}{\m}} x}^{-\frac{k}{\m-1}+1} \dx^k \ue \right\|_{L^2} \lesssim \| u(t) \|_{\wt{X}}. \label{uepL2}
\end{align}
\end{cor}

\begin{proof}
By \eqref{freqXt} and \eqref{est:uhperr}, we have
\begin{align*}
& \sum_{k=0}^{\m-2} \sum_{l=0}^k \left\| t^{\frac{k+1}{\m-1}} |x|^{-\frac{\m k+1}{\m-1}+l} \dx^l \uhp \right\|_{L^2} \\
& \lesssim \sum_{k=0}^{\m-2} \sum_{l=0}^k \Bigg( \sum _{\substack{N \in 2^{\mathbb{Z}} \\ N > t^{-\frac{1}{\m}}}} \left\| t^{\frac{k+1}{\m-1}} |x|^{-\frac{\m k+1}{\m-1}+l} \dx^l \uhp_N (t) \right\| _{L^2}^2 \Bigg)^{\frac{1}{2}} \\
& \quad + \sum_{k=0}^{\m-2} \sum_{l=0}^k \sum _{\substack{N \in 2^{\mathbb{Z}} \\ N > t^{-\frac{1}{\m}}}} \left\| \left( 1-\PN \right) t^{\frac{k+1}{\m-1}} |x|^{-\frac{\m k+1}{\m-1}+l} \dx^l \uhp_N \right\| _{L^2} \\
& \lesssim \sum_{k=0}^{\m-2} \sum_{l=0}^k \Bigg( \sum _{\substack{N \in 2^{\mathbb{Z}} \\ N > t^{-\frac{1}{\m}}}} \left\| t^{-k+l} N^{-\m k+ (\m-1) l-1} \dx^l \uhp_N (t) \right\| _{L^2}^2 \Bigg)^{\frac{1}{2}} + \| u(t) \|_{\wt{X}} \\
& \lesssim \| u (t) \|_{\wt{X}}.
\end{align*}

We use \eqref{freqXt}, \eqref{est:uhperr}, and \eqref{uhpNl2} to obtain
\begin{align*}
& \sum_{k=0}^{\m-2} \left\| t^{\frac{k}{\m-1}} |x|^{-\frac{k-\m+2}{\m-1}} \dx^k \J_+ \uhp \right\| _{L^2} \\
& \lesssim \sum_{k=0}^{\m-2} \Bigg( \sum _{\substack{N \in 2^{\mathbb{Z}} \\ N > t^{-\frac{1}{\m}}}} \left\| t^{\frac{k}{\m-1}} |x|^{-\frac{k-\m+2}{\m-1}} \dx^k \J_+ \uhp_N (t) \right\| _{L^2}^2 \Bigg)^{\frac{1}{2}} \\
& \quad + \sum_{k=0}^{\m-2} \sum _{\substack{N \in 2^{\mathbb{Z}} \\ N > t^{-\frac{1}{\m}}}} \left\| \left( 1-\PN \right) t^{\frac{k}{\m-1}} |x|^{-\frac{k-\m+2}{\m-1}} \dx^k \J_+ \uhp_N (t) \right\| _{L^2} \\
& \lesssim \| u (t) \|_{\wt{X}}.
\end{align*}
This gives \eqref{uhpL2} with $k=0$.
For $k \ge 1$, because
\begin{align*}
&\J_+ \dx^k \uhp \\
&= \dx^k \left( \J_+ \uhp \right) + t^{-\frac{k}{\m-1}} |x|^{\frac{k-\m+2}{\m-1}} \sum_{l=0}^{k-1} C_{k,l} t^{\frac{(k-1)+1}{\m-1}} |x|^{-\frac{\m (k-1)+1}{\m-1}+l} \dx^l \uhp,
\end{align*}
the estimate \eqref{uhpL2} follows from \eqref{uhpNl2} and \eqref{uhpL22}.

For the elliptic bound, we note that
\begin{equation} \label{uepw}
\ue = \uel + 2\Re \sum _{\substack{N \in 2^{\mathbb{Z}} \\ N > t^{-\frac{1}{\m}}}} \uep_N.
\end{equation}
We use \eqref{freqXt}, \eqref{est:ueperr}, and \eqref{uepNl2} to obtain
\begin{align*}
\sum_{k=0}^{\m-2} \left\| t^{\frac{k+1}{\m}} \dx^k \ue \right\|_{L^2}
& \lesssim t^{\frac{1}{\m}} \| \uel \|_{L^2} + \sum_{k=0}^{\m-2} \Bigg( \sum _{\substack{N \in 2^{\mathbb{Z}} \\ N > t^{-\frac{1}{\m}}}} \left( t^{\frac{k+1}{\m}} N^k \left\| \uep_N \right\|_{L^2} \right)^2 \Bigg)^{\frac{1}{2}} \\
& \quad + \sum_{k=0}^{\m-2} \sum _{\substack{N \in 2^{\mathbb{Z}} \\ N > t^{-\frac{1}{\m}}}} t^{\frac{k+1}{\m}} \left\|  \left( 1-\PN \right) \dx^k \uep_N \right\|_{L^2} \\
& \lesssim \| u(t) \|_{\wt{X}}.
\end{align*}
From $t^{\frac{k}{\m-1}} |x|^{-\frac{k}{\m-1}+1} N^k \lesssim |x| + tN^{\m-1}$ and \eqref{est:ueperr2}, we have
\begin{align*}
t^{\frac{k}{\m-1}} \left\| |x|^{-\frac{k}{\m-1}+1} \dx^k \uep_N \right\|_{L^2 (|x| \ge t^{\frac{1}{\m}})}
& \lesssim N^{-k} \left\| (|x|+tN^{\m-1}) \sigma_{>t^{\frac{1}{\m}}} (x) \dx^k \uep_N \right\|_{L^2} \\
& \lesssim \left\| (|x|+ tN^{\m-1}) \uep_N \right\|_{L^2} + t^{-\frac{1}{\m}} N^{-2} \| u_N \|_{L^2}.
\end{align*}
Because
\begin{align*}
t^{\frac{k}{\m-1}} \Big\| |x|^{-\frac{k}{\m-1}+1} \dx^k \uel \Big\|_{L^2 (|x| \ge t^{\frac{1}{\m}})} 
& \lesssim t^{\frac{k}{\m}} \Big\| x \dx^k \uel \Big\|_{L^2 (|x| \ge t^{\frac{1}{\m}})} \\
& \lesssim \Big\| \J \uel \Big\| + t^{\frac{1}{\m}} \Big\| \uel \Big\|_{L^2} \\
& \lesssim \| u(t) \|_{\wt{X}},
\end{align*}
by \eqref{freqXt}, \eqref{est:ueperr2}, and \eqref{uepNl2}, we obtain
\begin{align*}
& \left\| t^{\frac{k+1}{\m}} (t^{-\frac{1}{\m}} |x|)^{-\frac{k}{\m-1}+1} \dx^k \ue \right\|_{L^2 (|x| \ge t^{\frac{1}{\m}})} \\
& \lesssim \| u(t) \|_{\wt{X}} + \Bigg( \sum _{\substack{N \in 2^{\mathbb{Z}} \\ N > t^{-\frac{1}{\m}}}} \left\| (|x|+ tN^{\m-1}) \uep_N \right\|_{L^2}^2 \Bigg)^{\frac{1}{2}} + \sum _{\substack{N \in 2^{\mathbb{Z}} \\ N > t^{-\frac{1}{\m}}}} t^{-\frac{1}{\m}} N^{-2} \| u_N \|_{L^2} \\
& \quad + \sum _{\substack{N \in 2^{\mathbb{Z}} \\ N > t^{-\frac{1}{\m}}}} \left\| \left( 1-\PN \right) t^{\frac{k+1}{\m}} (t^{-\frac{1}{\m}} |x|)^{-\frac{k}{\m-1}+1} \sigma_{>t^{\frac{1}{\m}}} (x) \dx^k \uep_N \right\|_{L^2} \\
& \lesssim \| u(t) \|_{\wt{X}}. \qedhere
\end{align*}
\end{proof}

Second, we show the pointwise decay estimates.

\begin{prop} \label{prop:est|u|}
For $t \ge 1$ and $k=0,1, \dots, \m-2$, we have
\begin{align}
& \left| t^{\frac{k+1}{\m}} \lr{t^{-\frac{1}{\m}} x}^{-\frac{k}{\m-1}+\frac{\m-3}{2(\m-1)}} \dx^k \uhp (t,x) \right| \lesssim t^{-\frac{1}{2 \m}} \| u (t) \|_{\wt{X}}, \label{est:uhpp} \\
& \left| t^{\frac{k+1}{\m}} \lr{t^{-\frac{1}{\m}} x}^{-\frac{k}{\m-1}+\frac{2\m-3}{2(\m-1)}} \dx^k \ue (t,x) \right| \lesssim t^{-\frac{1}{2 \m}} \| u (t) \|_{\wt{X}}. \label{est:uepp}
\end{align}
\end{prop}

\begin{proof}
The Gagliardo-Nirenberg inequality
\[
|f| \lesssim \| f \|_{L^2}^{\frac{1}{2}} \| \dx f \|_{L^2}^{\frac{1}{2}}
\]
with $f = e^{-i\phi} \uhp_{N}$, $\dx (e^{-i\phi} \uhp) = -it^{-\frac{1}{\m-1}} e^{-i\phi} \J_+ \uhp$, Lemma \ref{lem:he_freq_est}, and \eqref{est:uhperr} imply
\begin{align*}
& \left| t^{\frac{k+1}{\m}} \lr{t^{-\frac{1}{\m}} x}^{-\frac{k}{\m-1}+\frac{\m-3}{2(\m-1)}} \dx^k \uhp_N (t,x) \right| \\
& \lesssim t^{\frac{\m-1}{2\m}} N^{\frac{\m-3}{2}-k} \left\| \dx^k \uhp_N (t) \right\|_{L^{\infty}} \\
& \lesssim t^{\frac{\m^2-3\m+1}{2\m (\m-1)}} N^{\frac{\m-3}{2}-k} \left\| \dx^k \uhp_{N} (t) \right\|_{L^2}^{\frac{1}{2}} \left\| \J_+ \dx^k \uhp_{N} (t) \right\|_{L^2}^{\frac{1}{2}} \\
& \lesssim t^{-\frac{1}{2\m}} \| N^{-1} u_{N} (t) \|_{L^2}^{\frac{1}{2}} \left\| t^{\frac{\m-2}{\m-1}} N^{\m-2} \J_+ \uhp_{N} (t) \right\|_{L^2}^{\frac{1}{2}} + t^{-\frac{1}{2\m}} \| u(t) \|_{\wt{X}} \\
& \lesssim t^{-\frac{1}{2\m}} \| u (t) \|_{\wt{X}}.
\end{align*}
Because $\uhp (t,x)$ is a finite sum of $\uhp_N(t,x)$'s, this yields the desired hyperbolic bound \eqref{est:uhpp}.

Next, we show the elliptic bound.
First, we consider the low frequency part.
For $|x| \le t^{\frac{1}{\m}}$, Bernstein's inequality implies
\[
\left| t^{\frac{k+1}{\m}} \lr{t^{-\frac{1}{\m}} x}^{-\frac{k}{\m-1}+\frac{2\m-3}{2(\m-1)}} \dx^k \uel (t,x) \right|
\lesssim t^{\frac{1}{2\m}} \Big\| \uel (t) \Big\|_{L^2}
\lesssim t^{-\frac{1}{2\m}} \| u(t) \|_{\wt{X}}.
\]
Similarly, for $|x| \ge t^{\frac{1}{\m}}$, we have
\begin{align*}
& \left| t^{\frac{k+1}{\m}} \lr{t^{-\frac{1}{\m}} x}^{-\frac{k}{\m-1}+\frac{2\m-3}{2(\m-1)}} \dx^k \uel (t,x) \right| \\
& \lesssim t^{\frac{k}{\m}} \Big\| x \dx^k \uel (t) \Big\|_{L^{\infty}} \\
& \lesssim t^{-\frac{1}{2\m}} \Big\| \J \uel (t) \Big\|_{L^2} + t^{\frac{1}{2\m}} \Big\| \uel (t) \Big\|_{L^2} \\
& \lesssim t^{-\frac{1}{2\m}} \| u(t) \|_{\wt{X}}.
\end{align*}

Second, we consider the high frequency part.
For $|x|\le t^{\frac{1}{\m}}$, the Gagliardo-Nirenberg inequality and \eqref{est:ueperr} yield
\begin{align*}
& \left| t^{\frac{k+1}{\m}} \lr{t^{-\frac{1}{\m}} x}^{-\frac{k}{\m-1}+\frac{2\m-3}{2(\m-1)}} \dx^k \uep_N (t,x) \right| \\
& \lesssim t^{\frac{k+1}{\m}} N^k \left\| \uep_N (t,x) \right\|_{L^2} + t^{\frac{1}{2\m}} (t^{\frac{1}{\m}}N)^{-2} \| u_N(t) \|_{L^2} \\
& = t^{\frac{k-\m+1}{\m}} N^{k-\m+1} \left\| tN^{\m-1}\uep_N (t,x) \right\|_{L^2} + t^{\frac{1}{2\m}} (t^{\frac{1}{\m}}N)^{-2} \| u_N(t) \|_{L^2}.
\end{align*}
Thus, by \eqref{uepNl2}, we have
\[
\sum_{\substack{N \in 2^{\mathbb{Z}} \\ N > t^{-\frac{1}{\m}}}} \left| t^{\frac{k+1}{\m}} \lr{t^{-\frac{1}{\m}} x}^{-\frac{k}{\m-1}+\frac{2\m-3}{2(\m-1)}} \dx^k \uep_N (t,x) \right|
\lesssim t^{-\frac{1}{2\m}} \| u(t) \|_{\wt{X}}.
\]
For $|x|>t^{\frac{1}{\m}}$, there exists $M \in 2^{\Z}$ such that
\[
\uep_N (t,x) = \sigma_{tM^{\m-1}}(x) \uep_N (t,x).
\]
Then, the Gagliardo-Nirenberg inequality and Lemma \ref{lem:freq_spat_loc} lead to
\begin{align*}
& \left| t^{\frac{k+1}{\m}} \lr{t^{-\frac{1}{\m}} x}^{-\frac{k}{\m-1}+\frac{2\m-3}{2(\m-1)}} \dx^k \uep_N (t,x) \right| \\
& \lesssim t^{\frac{2\m-1}{2\m}} M^{\frac{2\m-3}{2}-k} \left| \sigma_{tM^{\m-1}}(x) \dx^k \uep_N (t,x) \right| \\
& \lesssim t^{\frac{2\m-1}{2\m}} M^{\frac{2\m-3}{2}-k} N^{k+\frac{1}{2}} \left\| \sigma_{tM^{\m-1}}(x) \uep_N (t) \right\|_{L^2} \\
& \quad + t^{-k+\frac{2\m-1}{2\m}} N^{\frac{1}{2}} M^{\frac{2\m-3}{2}-\m k} \max \left( 1, tNM^{\m-1} \right)^{-2} \| u_N (t) \|_{L^2}.
\end{align*}
Hence, we have
\begin{align*}
& \sum_{\substack{N \in 2^{\mathbb{Z}} \\ N > t^{-\frac{1}{\m}}}} \left| t^{\frac{k+1}{\m}} \lr{t^{-\frac{1}{\m}} x}^{-\frac{k}{\m-1}+\frac{2\m-3}{2(\m-1)}} \dx^k \uep_N (t,x) \right| \\
& \lesssim \sum_{\substack{N \in 2^{\mathbb{Z}} \\ t^{-\frac{1}{\m}} < N \le M}} t^{-\frac{1}{2\m}} M^{-k-\frac{1}{2}} N^{k+\frac{1}{2}} \left\| x \uep_N (t) \right\|_{L^2} \\
& \quad + \sum_{\substack{N \in 2^{\mathbb{Z}} \\ N > M}} t^{-\frac{1}{2\m}} M^{\frac{2\m-3}{2}-k} N^{-\frac{2\m-3}{2}+k} \left\| tN^{\m-1} \uep_N (t) \right\|_{L^2}
+ t^{-\frac{1}{2\m}} \| u(t) \|_{\wt{X}} \\
& \lesssim t^{-\frac{1}{2\m}} \| u(t) \|_{\wt{X}}. \qedhere
\end{align*}
\end{proof}

\begin{rmk} \label{rmk:initialptb}
For $t \ge 1$, the estimate
\[
\left| t^{\frac{k}{\m}+\frac{3}{4\m}} \lr{t^{-\frac{1}{\m}} x}^{-\frac{k}{\m-1}+\frac{\m-2}{2(\m-1)}} \dx^k \uhp (t,x) \right|
\lesssim \| u(t) \|_{L^2} + t^{-\frac{1}{2\m}} \| u (t) \|_{\wt{X}}
\]
holds true.
Indeed, the Gagliardo-Nirenberg inequality, \eqref{est:uhperr}, and \eqref{uhpNl2} yield
\begin{align*}
& \left| t^{\frac{k}{\m}+\frac{3}{4\m}} \lr{t^{-\frac{1}{\m}} x}^{-\frac{k}{\m-1}+\frac{\m-2}{2(\m-1)}} \dx^k \uhp_N (t,x) \right| \\
& \lesssim t^{\frac{2\m-1}{4\m}} N^{\frac{\m-2}{2}-k} \left\| \dx^k \uhp_N (t) \right\|_{L^{\infty}} \\
& \lesssim t^{\frac{2\m^2-5\m+1}{4\m (\m-1)}} N^{\frac{\m-2}{2}-k} \left\| \dx^k \uhp_{N} (t) \right\|_{L^2}^{\frac{1}{2}} \left\| \J_+ \dx^k \uhp_{N} (t) \right\|_{L^2}^{\frac{1}{2}} \\
& \lesssim t^{-\frac{1}{4\m}} \| u_{N} (t) \|_{L^2}^{\frac{1}{2}} \left\| t^{\frac{\m-2}{\m-1}} N^{\m-2} \J_+ \uhp_{N} (t) \right\|_{L^2}^{\frac{1}{2}} + t^{-\frac{1}{2\m}} \| u(t) \|_{\wt{X}} \\
& \lesssim \| u(t) \|_{L^2} + t^{-\frac{1}{2\m}} \| u (t) \|_{\wt{X}}.
\end{align*}
Accordingly, by combining this estimate with \eqref{est:uepp} and Remark \ref{rmk:reqH2}, we obtain \eqref{est:u_infty} at $t=1$.
\end{rmk}

\section{Wave packets} \label{S:wave_packet}

\subsection{Construction of wave packets}

Let $t \ge 1$.
Setting
\begin{equation} \label{lambda}
\lambda := t^{-\frac{1}{2}} v^{-\frac{\m-2}{2(\m-1)}}
=t^{-\frac{1}{\m}} (t^{\frac{\m-1}{\m}} v)^{-\frac{\m-2}{2(\m-1)}},
\end{equation}
we define, for $v \ge t^{-\frac{\m-1}{\m}}$,
\[
\Psi _v(t,x) := \chi \left( \lambda (x-vt) \right) e^{i \phi (t,x)},
\]
where $\chi$ is a smooth function with $\supp \chi \subset [-\frac{1}{2},\frac{1}{2}]$ and $\int_{\R} \chi (z) dz =1$, and $\phi$ is defined by \eqref{phase}.
The spatial support of $\Psi_v$ is included in $[\frac{vt}{2}, \frac{3}{2}vt]$.

We show that $\Psi_v(t,x)$ is essentially localized at frequency $\xi_v := v^{\frac{1}{\m-1}}$ in the following sense:

\begin{lem} \label{lem:freq_psi}
For $t \ge 1$ and $v \ge t^{-\frac{\m-1}{\m}}$, we have
\[
\mathcal{F}[ \Psi_v] (t,\xi)
= \frac{1}{\sqrt{\m-1}} \lambda^{-1} \chi_1 \left( \lambda^{-1} (\xi -\xi_v), \lambda^{-1} \xi_v \right) e^{-\frac{1}{\m}it \xi^{\m}},
\]
where $\chi_1 (\cdot, \al) \in \mathcal{S} (\R)$ satisfies
\begin{equation} \label{est:chi1p}
\sup_{\al \ge 1} \sup_{ \zeta \in \R} \left| \lr{\zeta}^k \dz^l \chi_1 (\zeta, \al) \right| \lesssim_{k,l} 1
\end{equation}
for any $k,l \in \N_0$.
Moreover, there exists a constant $C_1>0$ such that for any $\al \ge 1$,
\begin{equation} \label{chi11}
\left| \int_{\R} \chi_1(\zeta, \al) d\zeta - 1 \right| \le \frac{C_1}{\al}.
\end{equation}
\end{lem}

\begin{proof}
From Taylor's theorem, we can write for $x>0$
\begin{align*}
&\phi (t,x) \\
& = \phi (t,vt) + \dx \phi (t,vt) (x-vt) + \frac{1}{2} \dx^2 \phi (t,vt) (x-vt)^2 + \int_{vt}^x \frac{(x-y)^2}{2} \dx^3 \phi (t,y) dy \\
& = -\frac{\pi}{4} + \frac{\m-1}{\m} t \xi_v^{\m} + \xi_v (x-vt) + \frac{1}{2(\m-1)} \lambda^2 (x-vt)^2 + R \left( \lambda (x-vt), \lambda^{-1} \xi_v \right) ,
\end{align*}
where
\[
R(z,\al) := -\frac{\m-2}{2(\m-1)^2} \frac{z^3}{\al} \int _0^1 (1-\theta)^2 \left( \theta \frac{z}{\al}+1 \right)^{-\frac{2\m-3}{\m-1}} d\theta .
\]
We note that $R (z,\al)$ is well-defined provided that $z>-\al$.
By a change of variables using $z= \lambda (x-vt)$, we have
\begin{align*}
& \mathcal{F} [\Psi_v] (t,\xi) \\
& = \frac{1}{\sqrt{2\pi}} \int_{\R} e^{-ix\xi} \chi ( \lambda (x-vt)) e^{i \phi (t,x)} dx \\
& = e^{-i \frac{\pi}{4}} \lambda^{-1} e^{\frac{1}{\m} it (-\m \xi \xi_v^{\m-1} + (\m-1) \xi_v^{\m})} \frac{1}{\sqrt{2\pi}} \int _{\R} e^{-i z \lambda^{-1} (\xi - \xi_v)} e^{\frac{i}{2(\m-1)} z^2 + i R(z,  \lambda^{-1} \xi_v)} \chi (z) dz  \\
& = \frac{1}{\sqrt{\m-1}} \lambda^{-1} \chi_1 \left( \lambda^{-1} (\xi - \xi_v), \lambda^{-1} \xi_v \right) e^{-\frac{1}{\m}it\xi^{\m}} ,
\end{align*}
where
\[
\chi_1 (\zeta, \al) := (1-i) \sqrt{\frac{\m-1}{2}} e^{\frac{i}{\m} \sum_{l=0}^{\m-2} \left( \begin{smallmatrix} \m \\ l \end{smallmatrix} \right) \zeta^{\m-l} \al^{-\m+l+2}} \mathcal{F} [e^{\frac{i}{2(\m-1)} z^2 + i R(z, \al)} \chi ] (\zeta).
\]
By definition, $\chi_1 (\cdot, \al) \in \mathcal{S}(\R)$ for $\al \ge 1$.
From the Fresnel integrals, 
\[
(1-i) \sqrt{\frac{\pi}{2}} = \int_{\R} e^{-i \left( \eta +\sqrt{\frac{\m-1}{2}} (\zeta-\frac{z}{\m-1}) \right)^2} d\eta
\]
holds for any $z, \zeta \in \R$.
Accordingly, we have
\begin{align*}
& \mathcal{F} [e^{\frac{i}{2(\m-1)} z^2 + i R(z, \al)} \chi ] (\zeta) \\
& = e^{-i \frac{m-1}{2} \zeta^2} \frac{1}{\sqrt{2\pi}} \int_{\R} e^{i \frac{\m-1}{2} (\zeta-\frac{z}{\m-1})^2 + iR(z,\al)} \chi (z) dz \\
& = \frac{1}{1-i} \sqrt{\frac{2}{\pi}} e^{-\frac{\m-1}{2} i\zeta^2} \frac{1}{\sqrt{2\pi}} \int_{\R} \int_{\R} e^{-i \left( \eta^2 + \sqrt{2 (\m-1)} \eta (\zeta - \frac{z}{\m-1}) \right)} e^{iR(z,\al)} \chi (z) d\eta dz \\
& = \frac{1}{1-i} \sqrt{\frac{2}{\pi}} e^{-\frac{\m-1}{2}i\zeta^2} \int_{\R} e^{-i\eta^2} e^{-i \sqrt{2 (\m-1)} \eta \zeta} \wh{\chi_2} \left( -\sqrt{\frac{2}{\m-1}} \eta, \al \right) d\eta,
\end{align*}
where
\[
\chi_2 (\cdot, \al) := \chi e^{iR(\cdot,\al)} \in \mathcal{S} (\R)
\]
for $\al \ge 1$.
Hence, we can write
\begin{equation} \label{eq:chi1}
\chi_1 (\zeta, \al)
= \frac{\m-1}{2} \sqrt{\frac{2}{\pi}} e^{\frac{i}{\m} \sum_{l=0}^{\m-3} \left( \begin{smallmatrix} \m \\ l \end{smallmatrix} \right) \zeta^{\m-l} \al^{-\m+l+2}} \int_{\R} e^{-i \frac{\m-1}{2} \eta^2} e^{i (\m-1) \eta \zeta} \wh{\chi_2} \left( \eta, \al \right) d\eta.
\end{equation}
As $\chi_2 (\cdot, \al) \in \mathcal{S} (\R)$, $e^{\frac{i}{\m} \sum_{l=0}^{\m-3} \left( \begin{smallmatrix} \m \\ l \end{smallmatrix} \right) \zeta^{\m-l} \al^{-\m+l+2}} = 1 + O \left( \frac{|\zeta|^3}{\al} \lr{\zeta}^{\m-3} \right)$, and
\begin{equation} \label{eq:chi2}
\sup_{\al \ge 1} \sup_{\eta \in \R} |\lr{\eta}^k \de^l \wh{\chi_2} (\eta, \al)| \lesssim_{k,l} 1,
\end{equation}
we obtain
\[
\int_{\R} \chi_1 (\zeta,\al) d\zeta
= \sqrt{2\pi} \wh{\chi_2} (0,\al) +O \left( \frac{1}{\al} \right).
\]
Because $e^{i R(z,\al)} = 1 + O (\frac{1}{\al})$ for $|z| < 1$ and $\al \ge 1$, we have
\[
\sqrt{2\pi} \wh{\chi_2} (0,\al) = \int_{\R} \chi (z) e^{iR(z,\al)} dz = 1+ O \left( \frac{1}{\al} \right).
\]
Finally, \eqref{est:chi1p} follows from \eqref{eq:chi1} and \eqref{eq:chi2}.
\end{proof}

For $v \ge t^{-\frac{\m-1}{\m}}$, we define the nearest dyadic number to $\xi_v$ by $N_{v} \in 2^{\mathbb{Z}}$.
Then, $\frac{\xi_v}{2} < N_v < 2\xi_v$ holds.

Integration by parts with \eqref{est:chi1p} yields
\[
\left| \left( 1- \PNv^+ \right) |\dx|^a \Psi_v(t,x) \right|
\lesssim_{a,l} t^{-\frac{a}{\m}} \left( t^{\frac{\m-1}{\m}} v \right)^{-l} \min (1, |x|^{-1} tv)^2
\]
for any $a, l \ge 0$, which implies
\begin{equation}
\label{Psi_L1}
\left\| \left( 1- \PNv^+ \right) |\dx|^a \Psi_v(t) \right\|_{L_x^1}
\lesssim_{a,c} t^{\frac{1-a}{\m}} \left( t^{\frac{\m-1}{\m}} v \right)^{-c}
\end{equation}
for $v \ge t^{-\frac{\m-1}{\m}}$ and any $a, c \ge 0$.

Next, we show that $\Psi_v$ is a good approximate solution for the linear equation.

We begin with the following preliminary observation.
Set
\[
S_j := \{ k= (k_1, \dots, k_j) \in \N_0 \colon 0 \le k_1 \le \dots \le k_j \le j, \, k_1+ \dots + k_j=j \}
\]
for $j \in \mathbb{N}$.
\begin{lem} \label{lem:diff}
Let $f$ be a smooth function.
For any $j \in \mathbb{N}$, we have
\begin{equation} \label{eq:diff}
\frac{d^j}{dx^j} e^{f(x)} = \sum_{k \in S_j} C^{(j)}_k e^{f(x)} \prod_{l=1}^j f^{(k_l)}(x),
\end{equation}
where
\[
f^{(l)} := \begin{cases} 1, & \text{if } l=0, \\ \frac{d^l f}{dx^l}, & \text{otherwise}, \end{cases}
\]
and $C^{(j)}_k$ is a constant depending on $j \in \mathbb{N}$, $k \in S_j$.
In particular,
\[
C^{(j)}_{(1,\dots,1)} = 1, \quad C^{(j)}_{(0,1,\dots,1,2)} = \begin{pmatrix} j \\ 2 \end{pmatrix}.
\]
\end{lem}

\begin{proof}
A direct calculation shows
\begin{align*}
\frac{d}{dx} e^{f(x)} &= e^{f(x)} f'(x), \qquad
\frac{d^2}{dx^2} e^{f(x)} = e^{f(x)} \left( f'(x)^2 + f''(x) \right), \\
\frac{d^3}{dx^3} e^{f(x)} &= e^{f(x)} \left( f'(x)^3 + 3f'(x) f''(x) + f'''(x) \right).
\end{align*}
Hence, we have
\begin{equation} \label{initialC}
C^{(1)}_{(1)} = C^{(2)}_{(1,1)} = C^{(2)}_{(0,2)} = C^{(3)}_{(1,1,1)} = C^{(3)}_{(0,0,3)}=1, \quad C^{(3)}_{(0,1,2)}=3.
\end{equation}
We assume that \eqref{eq:diff} holds up to $j-1$.
Because
\begin{align*}
&\frac{d^j}{dx^j} e^{f(x)} \\
&= \frac{d}{dx} \left( \sum_{k \in S_{j-1}} C^{(j-1)}_k e^{f} \prod_{l=1}^{j-1} f^{(k_l)} \right) (x) \\
&= \sum_{k \in S_{j-1}} C^{(j-1)}_k e^{f(x)} \Bigg( f'(x) \prod_{l=1}^{j-1} f^{(k_l)}(x) + \sum_{\substack{n \in \{ 1, \dots, j-1\} \\ k_n \neq 0}} f^{(k_n+1)}(x) \prod_{l \neq n} f^{(k_l)}(x) \Bigg),
\end{align*}
the constants $C^{(j)}_k$ are determined by $C^{(j-1)}_k$, which shows \eqref{eq:diff}.

In particular, the following recurrence equations hold true:
\[
C^{(j)}_{(1, \dots , 1)} = C^{(j-1)}_{(1,\dots,1)}, \quad
C^{(j)}_{(0,1,\dots,1,2)} = C^{(j-1)}_{(0,1,\dots,1,2)} + (j-1) C^{(j-1)}_{(1,\dots,1)}.
\]
By \eqref{initialC}, we obtain $C^{(j)}_{(1,\dots,1)} = 1$ and  $C^{(j)}_{(0,1,\dots,1,2)} = \begin{pmatrix} j \\ 2 \end{pmatrix}$.
\end{proof}

Set
\[
S_j' := S_j \setminus \{ (1,\dots, 1) \}, \quad
S_j'' := S_j' \setminus \{ (0,1, \dots, 1,2) \}.
\]
Lemma \ref{lem:diff} and $\dt \phi = -\frac{1}{\m} (\dx \phi)^{\m}$ yield
\begin{align*}
& \left( \dt + \frac{i}{\m} (-i \dx)^{\m} \right) \Psi_v (t,x) \\
& = -\frac{x+vt}{2t} \lambda \chi' e^{i\phi} + i \dt \phi \chi e^{i\phi} \\
& \quad + \frac{i}{\m} (-i)^{\m} \bigg\{ \chi \dx^{\m} \left( e^{i \phi} \right) + \m \lambda \chi' \dx^{\m-1} \left( e^{i \phi} \right)  + \frac{\m (\m-1)}{2} \lambda^2 \chi'' \dx^{\m-2} \left( e^{i \phi}\right) \\
& \quad + \frac{\m (\m-1)(\m-2)}{6} \lambda^3 \chi''' \dx^{\m-3} \left( e^{i \phi}\right)
+ \sum_{j=0}^{\m-4} \begin{pmatrix} \m \\ j \end{pmatrix} \lambda^{\m-j} \chi^{(\m-j)} \dx^j \left( e^{i \phi} \right) \bigg\} \\
& = -\frac{x+vt}{2t} \lambda \chi' e^{i\phi} + i \dt \phi \chi e^{i\phi} \\
& \quad +\frac{i}{\m} \chi \bigg\{ (\dx \phi)^{\m} - i \begin{pmatrix} \m \\ 2 \end{pmatrix} (\dx \phi)^{\m-2} \dx^2 \phi + (-i)^{\m} \sum_{k \in S''_{\m}} C^{(\m)}_k \prod_{l=1}^{\m} (i \phi)^{(k_l)} \bigg\} e^{i\phi} \\
& \quad + \lambda \chi' \bigg\{ (\dx \phi)^{\m-1} -i \begin{pmatrix} \m-1 \\ 2 \end{pmatrix} (\dx \phi)^{\m-3} \dx^2 \phi \\
& \hspace*{40pt} + (-i)^{\m-1} \sum_{k \in S''_{\m-1}} C^{(\m-1)}_k \prod_{l=1}^{\m-1} (i \phi)^{(k_l)} \bigg\} e^{i\phi} \\
& \quad -i \frac{\m-1}{2} \lambda^2 \chi'' \bigg\{ (\dx \phi)^{\m-2} -i \begin{pmatrix} \m-2 \\2 \end{pmatrix} (\dx \phi)^{\m-4} \dx^2 \phi \\
& \hspace*{80pt} + (-i)^{\m-2} \sum_{k \in S''_{\m-2}} C^{(\m-2)}_k \prod_{l=1}^{\m-2} (i\phi)^{(k_l)} \bigg\} e^{i\phi} \\
& \quad - \frac{(\m-1)(\m-2)}{6} \lambda^3 \chi''' \bigg\{ (\dx \phi)^{\m-3} + (-i)^{\m-3} \sum_{k \in S'_{\m-3}} C^{(\m-3)}_k \prod_{l=1}^{\m-3} (i\phi)^{(k_l)} \bigg\} e^{i\phi} \\
& \quad + \frac{i}{\m} (-i)^{\m} \sum_{j=0}^{\m-4} \begin{pmatrix} \m \\ j \end{pmatrix} \lambda^{\m-j} \chi^{(\m-j)} \dx^j \left( e^{i \phi} \right) \\
& = \bigg\{ \dx \left( \frac{x-vt}{2t} \chi \right) - i \frac{\m-1}{2} t^{-\frac{\m-2}{\m-1}} \dx \left( |\cdot|^{\frac{\m-2}{\m-1}} \lambda \chi' \right) \\
& \quad - \frac{(\m-1)(\m-2)}{6} t^{-\frac{\m-3}{\m-1}} \dx \left( |\cdot|^{\frac{\m-3}{\m-1}} \lambda^2 \chi'' \right) - \frac{(\m-2)(\m-3)}{12} t^{-\frac{\m-3}{\m-1}} |x|^{-\frac{2}{\m-1}} \lambda^2 \chi'' \\
& \quad +\frac{i}{\m}  (-i)^{\m} \chi \sum_{k \in S''_{\m}} C^{(\m)}_k \prod_{l=1}^{\m} (i \phi)^{(k_l)}
+ (-i)^{\m-1} \lambda \chi' \sum_{k \in S''_{\m-1}} C^{(\m-1)}_k \prod_{l=1}^{\m-1} (i \phi)^{(k_l)}  \\
& \quad -i \frac{\m-1}{2} (-i)^{\m-2} \lambda^2 \chi'' \sum_{k \in S''_{\m-2}} C^{(\m-2)}_k \prod_{l=1}^{\m-2} (i\phi)^{(k_l)} \\
& \quad - \frac{(\m-1)(\m-2)}{6} (-i)^{\m-3} \lambda^3 \chi''' \sum_{k \in S'_{\m-3}} C^{(\m-3)}_k \prod_{l=1}^{\m-3} (i\phi)^{(k_l)} \\
& \quad + \frac{i}{\m} (-i)^{\m} \sum_{j=0}^{\m-4} \begin{pmatrix} \m \\ j \end{pmatrix} \lambda^{\m-j} \chi^{(\m-j)} \sum_{k \in S_{j}} C^{(j)}_{k} \prod_{l=1}^j (i \phi)^{(k_l)} 
 \bigg\} e^{i\phi} \\
& = \frac{e^{i\phi}}{t \lambda} \dx \wt{\chi} + O \left( t^{-1} (t^{\frac{\m-1}{\m}} v)^{-\frac{\m}{\m-1}} \xc (\lambda (x-vt)) \right),
\end{align*}
where
\begin{align*}
\wt{\chi}(t,x) := & \lambda \frac{x-vt}{2} \chi (\lambda (x-vt)) - i \frac{\m-1}{2} \lambda^2 t^{\frac{1}{\m-1}} |x|^{\frac{\m-2}{\m-1}} \chi' (\lambda (x-vt)) \\ 
& - \frac{(\m-1)(\m-2)}{6} \lambda^3 t^{\frac{2}{\m-1}} |x|^{\frac{\m-3}{\m-1}} \chi''(\lambda (x-vt))
\end{align*}
and $\xc$ is a nonnegative continuous function supported in $[-\frac{1}{2}, \frac{1}{2}]$.
Therefore, we obtain the following:
\begin{equation} \label{eq:LPsi}
\begin{aligned}
(\Lc \Psi_v) (t,x) = & \frac{e^{i\phi}}{t \lambda} \dx \wt{\chi} + \frac{1}{\m} \left( |\dx|^{\m-1} \dx - i (-i \dx)^{\m} \right) P^- \Psi_v (t,x) \\
&+ O \left( t^{-1} (t^{\frac{\m-1}{\m}} v)^{-\frac{\m}{\m-1}} \xc (\lambda (x-vt)) \right).
\end{aligned}
\end{equation}
Because $\wt{\chi}$ has the same localization as $\chi$, the first term on the right hand side of \eqref{eq:LPsi} is essentially localized at frequency $\xi_v$.
For the sake of completeness, we give a proof here, although the proof is similar to that of Lemma \ref{lem:freq_psi}.

\begin{lem} \label{LPsi1}
For $t \ge 1$ and $v \ge t^{-\frac{\m-1}{\m}}$ and any $a, c \ge 0$, we have
\begin{equation}
\label{LPsi_L1}
\left\| \left( 1- \PNv^+ \right) |\dx|^a (e^{i\phi} \wt{\chi}) \right\|_{L_x^1}
\lesssim_{a,c} t^{\frac{1-a}{\m}} \left( t^{\frac{\m-1}{\m}} v \right)^{-c}.
\end{equation}
\end{lem}

\begin{proof}
We write
\[
\wt{\chi}(t,x) := \wt{\chi}_0(\lambda (x-vt), \lambda vt),
\]
where
\begin{align*}
\wt{\chi}_0 (z,\al) := & \frac{z}{2} \chi (z) - i \frac{\m-1}{2} \al^{-\frac{\m-2}{\m-1}} |z+\al|^{\frac{\m-2}{\m-1}} \chi' (z) \\
&- \frac{(\m-1)(\m-2)}{6} \al^{-\frac{2(\m-2)}{\m-1}} |z+\al|^{\frac{\m-3}{\m-1}} \chi''(z).
\end{align*}
The same calculation as in the proof of Lemma \ref{lem:freq_psi} yields
\[
\mathcal{F}[e^{i\phi} \wt{\chi}] (t,\xi)
:= \frac{1}{\sqrt{\m-1}} \lambda ^{-1} \wt{\chi}_1 \left( \lambda^{-1}(\xi-\xi_v), \lambda^{-1}\xi_v \right) e^{-\frac{1}{\m} it \xi^{\m}},
\]
where
\[
\wt{\chi}_1(\zeta,\al) := (1-i) \sqrt{\frac{\m-1}{2}} e^{\frac{i}{\m} \sum_{l=0}^{\m-2} \left( \begin{smallmatrix} \m \\ l \end{smallmatrix} \right) \zeta^{\m-l} \al^{-\m+l+2}} \mathcal{F} [e^{\frac{i}{2(\m-1)} z^2 + i R(z, \al)} \wt{\chi}_0(\cdot,\al) ] (\zeta).
\]
Since $\wt{\chi}_0(\cdot, \al) \in \Sc (\R)$ for $\al \ge 1$, we have
\[
\sup_{\al \ge 1} \sup_{ \zeta \in \R} |\lr{\zeta}^k \dz^l \wt{\chi}_1 (\zeta, \al)| \lesssim_{k,l} 1.
\]
From $|\xi - \xi_v| \ge \frac{\xi_v}{2}$ provided that $\xi \notin [ \frac{N_v}{4}, 4 N_v]$, integration by parts yields
\[
\left| \left( 1- \PNv^+ \right) |\dx|^a (e^{i\phi} \wt{\chi}) (x) \right|
\lesssim_l t^{-\frac{a}{\m}} \left( t^{\frac{\m-1}{\m}} v \right)^{-l} \min (1, |x|^{-1} tv)^2,
\]
which implies the desired bound.
\end{proof}

\subsection{Testing by wave packets}

Let $C_2>0$ be the constant appearing in \eqref{est:chi1p} with $k=2$ and $l=0$, that is,
\[
\sup_{\al \ge 1} \sup_{ \zeta \in \R} \left| \lr{\zeta}^2 \chi_1 (\zeta, \al) \right| \le C_2.
\]
For $t \ge 1$, we define
\[
\Ot := \left\{ v \in \R_+ \colon v \ge C_{\ast} t^{-\frac{\m-1}{\m}} \right\},
\]
where
\begin{equation} \label{constast}
C_{\ast} := \left( 2(C_1+C_2+1) \right)^{\frac{2(\m-1)}{\m}}.
\end{equation}
Here, $C_1$ is the constant appearing in \eqref{chi11}.
The large constant $C_{\ast}$ is needed to show the pointwise estimate \eqref{ugfs} in the frequency space.

We observe that the output $\gamma (t,v)$ defined by \eqref{output} is a ``good'' approximation of $u$ for $v \in \Ot$.

\begin{prop} \label{prop:approx}
For $t \ge 1$ and $k=0,1, \dots, \m-2$, we have the bounds
\begin{equation}
\dx^k u^+ (t,vt) = i^k \lambda v^{\frac{k}{\m-1}} e^{i \phi (t,vt)} \gamma (t,v) +R_k(t,v),
\label{ugsp1}
\end{equation}
where $R_k$ is a function satisfying
\begin{align*}
\left\| t^{\frac{k+1}{\m}} (t^{\frac{\m-1}{\m}} v)^{-\frac{k}{\m-1}+\frac{3(\m-2)}{4(\m-1)}} R_k(t,v) \right\|_{L^{\infty}_v (\Ot)}
&\lesssim t^{-\frac{1}{2\m}} \| u(t) \|_{\wt{X}}, \\
\left\| t^{\frac{k}{\m}+\frac{\m+1}{2\m}} (t^{\frac{\m-1}{\m}} v)^{-\frac{k}{\m-1}+\frac{\m-2}{2(\m-1)}} R_k (t,v) \right\|_{L^2_v (\Ot)}
&\lesssim t^{-\frac{1}{2\m}} \| u(t) \|_{\wt{X}}.
\end{align*}
Moreover, in the frequency space, we have
\begin{equation} \label{ugfs}
\wh{u} (t,\xi_v) = \sqrt{\m-1} e^{-\frac{1}{\m} it \xi_v^{\m}} \gamma (t,v) + R_{\xi} (t,v),
\end{equation}
where $R_{\xi}$ is a function satisfying
\begin{align*}
\left\| (t^{\frac{\m-1}{\m}} v)^{\frac{\m-2}{4(\m-1)}} R_{\xi} (t,v) \right\|_{L^{\infty}_v (\Ot)}
&\lesssim t^{-\frac{1}{2\m}} \| u(t) \|_{\wt{X}}, \\
\left\| t^{\frac{\m-1}{2\m}} R_{\xi} (t,v) \right\|_{L^2_v (\Ot)}
&\lesssim t^{-\frac{1}{2\m}} \| u(t) \|_{\wt{X}}.
\end{align*}
\end{prop}

\begin{proof}
First, we show that
\begin{equation} \label{eq:L^2v}
\left\| v^{-\frac{\m-2}{2(\m-1)}} \int_{\R} f(t,x) \chi (\lambda (x-vt)) dx \right\|_{L^2_v (\Ot)}
\lesssim \| f(t,\cdot) \|_{L^2_x ([t^{\frac{1}{\m}}, \infty))}
\end{equation}
holds true.
By a change of variables using $z = \lambda (x-vt)$,
\[
\text{L.H.S. of \eqref{eq:L^2v}}
= t^{\frac{1}{2}} \left\| \int_{\R} f \left( t, t^{\frac{1}{2}} v^{\frac{\m-2}{2(\m-1)}} z + vt \right) \chi (z) dz \right\|_{L^2_v(\Ot)}.
\]
Setting $\wt{v} = t^{\frac{1}{2}} v^{\frac{\m-2}{2(\m-1)}} z + vt$, we note that
\begin{align*}
t^{-\frac{1}{\m}} \wt{v} &= t^{\frac{\m-1}{\m}} v \Big\{ 1+ (t^{\frac{\m-1}{\m}} v)^{-\frac{\m}{2(\m-1)}} z \Big\} \ge 1, \\
\frac{d \wt{v}}{dv} &= t \Big\{ 1+\frac{\m-2}{2(\m-1)} (t^{\frac{\m-1}{\m}} v)^{-\frac{\m}{2(\m-1)}} z \Big\} \ge \frac{t}{2},
\end{align*}
for $v \in \Ot$ and $|z| \le \frac{1}{2}$.
Then, we have
\begin{align*}
\text{L.H.S. of \eqref{eq:L^2v}}
&\lesssim t^{\frac{1}{2}} \int_{\R} \left\| f \left( t, t^{\frac{1}{2}} v^{\frac{\m-2}{2(\m-1)}} z + vt \right) \right\|_{L^2_v(\Ot)} \chi (z) dz \\
&\lesssim \| f(t,\cdot) \|_{L^2_x ([t^{\frac{1}{\m}}, \infty))}.
\end{align*}

Second, we show that $u$ in the definition of $\gamma$ is replaced with $\uhp$ up to error terms;
\begin{equation} \label{gammaapp1}
i^k \lambda v^{\frac{k}{\m-1}} \gamma (t,v)
= i^k \lambda v^{\frac{k}{\m-1}} \int_{\R} \uhp (t,x) \ol{\Psi_v (t,x)} dx + R_k(t,v).
\end{equation}
In fact, Proposition \ref{prop:est|u|} and \eqref{Psi_L1} imply
\begin{equation} \label{estu-}
\Big| \int_{\R} \ol{u^+ (t,x)} \ol{\Psi_v (t,x)} dx \Big|
\le \| u \|_{L^{\infty}} \left\| P^- \Psi_v(t) \right\|_{L^1}
\lesssim (t^{\frac{\m-1}{\m}} v)^{-1} \cdot t^{-\frac{1}{2\m}} \| u(t) \|_{\wt{X}}.
\end{equation}
Moreover, \eqref{est:uepp} yields
\begin{align*}
\Big| \int_{\R} \ue (t,x) \ol{\Psi_v (t,x)} dx \Big|
& \lesssim \lambda^{-1} (t^{\frac{\m-1}{\m}}v)^{-\frac{2\m-3}{2(\m-1)}} \left\| (t^{-\frac{1}{\m}} |x|)^{\frac{2\m-3}{2(\m-1)}} \ue (t) \right\|_{L^{\infty}} \\
& \lesssim (t^{\frac{\m-1}{\m}} v)^{-\frac{1}{2}} \cdot t^{-\frac{1}{2\m}} \| u(t) \|_{\wt{X}}.
\end{align*}
In addition, by \eqref{eq:L^2v} and \eqref{uepL2}, we have
\begin{align*}
\left\| \int_{\R} \ue (t,x) \ol{\Psi_v (t,x)} dx \right\|_{L^2_v (\Ot)}
& \lesssim t^{-\frac{\m-2}{2(\m-1)}} \left\| |x|^{\frac{\m-2}{2(\m-1)}} \ue (t) \right\|_{L^2_x([t^{\frac{1}{\m}},\infty))} \\
& \lesssim t^{-\frac{\m-2}{2\m}} \left\| \lr{t^{-\frac{1}{\m}} x}^{\frac{\m-2}{2(\m-1)}} \ue (t) \right\|_{L^2} \\
& \lesssim t^{-\frac{\m-1}{2\m}} \cdot t^{-\frac{1}{2\m}} \| u(t) \|_{\wt{X}}.
\end{align*}
Hence, from $\lambda v^{\frac{k}{\m-1}} = t^{-\frac{k+1}{\m}} (t^{\frac{\m-1}{\m}} v)^{\frac{k}{\m-1}-\frac{\m-2}{2(\m-1)}}$, we obtain \eqref{gammaapp1}.

Third, we observe that the equation
\begin{equation} \label{gammaapp2}
i^k \lambda v^{\frac{k}{\m-1}} \gamma (t,v) 
= \lambda \int_{\R} \dx^k\uhp (t,x) \ol{\Psi_v (t,x)} dx + R_k (t,v).
\end{equation}
holds true.
We note that
\begin{align*}
\uhp (t,x) \ol{\Psi_v (t,x)}
=& -i v^{-\frac{1}{\m-1}} \dx \uhp (t,x) \ol{\Psi_v(t,x)}  \\
&- i t^{\frac{1}{\m-1}} \left( x^{-\frac{1}{\m-1}}- (vt)^{-\frac{1}{\m-1}} \right) \dx \uhp (t,x) \ol{\Psi_v (t,v)} \\
& +i t^{\frac{1}{\m-1}} x^{-\frac{1}{\m-1}} \dx ( \uhp e^{-i\phi}) (t,x) \chi (\lambda (x-vt)).
\end{align*}
Here, \eqref{est:uhpp} yields
\begin{align*}
& v^{-\frac{k-1}{\m-1}} \left| \int_{\R} t^{\frac{1}{\m-1}} \left( x^{-\frac{1}{\m-1}} - (vt)^{-\frac{1}{\m-1}} \right) \dx^k \uhp (t,x)  \ol{\Psi_v(t,x)} dx \right| \\
& \lesssim t^{\frac{k}{\m}} (t^{\frac{\m-1}{\m}} v)^{-\frac{k}{\m-1}-\frac{\m}{2(\m-1)}} \int_{\R} \left| \dx^k \uhp (t,x) \chi (\lambda (x-vt)) \right| dx \\
& \lesssim (t^{\frac{\m-1}{\m}} v)^{-\frac{1}{2}} \cdot t^{-\frac{1}{2\m}} \| u(t) \|_{\wt{X}}.
\end{align*}
By \eqref{eq:L^2v} and \eqref{uhpL22}, we have
\begin{align*}
& \left\| v^{-\frac{k-1}{\m-1}} \int_{\R} t^{\frac{1}{\m-1}} \left( x^{-\frac{1}{\m-1}} - (vt)^{-\frac{1}{\m-1}} \right) \dx^k \uhp (t,x)  \ol{\Psi_v(t,x)} dx \right\|_{L^2_v(\Ot)} \\
& \lesssim t^{-\frac{1}{2}} \left\| v^{-\frac{k}{\m-1}-\frac{\m}{2(\m-1)}} \int_{\R} \dx^k \uhp (t,x) \chi (\lambda (x-vt)) dx \right\|_{L^2_v(\Ot)} \\
& \lesssim t^{-\frac{1}{2}} \left\| \left( \frac{x}{t} \right)^{-\frac{k+1}{\m-1}} \dx^k \uhp \right\|_{L^2_x} \\
& \lesssim t^{-\frac{\m-1}{2\m}} \cdot t^{-\frac{1}{2\m}} \| u(t) \|_{\wt{X}}.
\end{align*}
Moreover, \eqref{uhpL2} implies
\begin{align*}
& \left| v^{-\frac{k}{\m-1}}\int_{\R} t^{\frac{1}{\m-1}} x^{-\frac{1}{\m-1}} \dx \left( e^{-i \phi} \dx^k\uhp \right) (t,x) \chi (\lambda (x-vt)) dx \right| \\
& \lesssim t^{\frac{k}{\m-1}} (tv)^{-1} \lambda^{-\frac{1}{2}} \left\| |x|^{-\frac{k-\m+2}{\m-1}} \J_+ \dx^k \uhp \right\|_{L^2} \\
& \lesssim (t^{\frac{\m-1}{\m}} v)^{-\frac{3\m-2}{4(\m-1)}} \cdot t^{-\frac{1}{2\m}} \| u(t) \|_{\wt{X}}
\end{align*}
and
\begin{align*}
& \left\| v^{-\frac{k}{\m-1}}\int_{\R} t^{\frac{1}{\m-1}} x^{-\frac{1}{\m-1}} \dx (e^{-i \phi} \dx^k\uhp) (t,x) \chi (\lambda (x-vt)) dx \right\|_{L^2_v(\Ot)} \\
& \lesssim \left\| (t^{\frac{\m-1}{\m}} v)^{-\frac{3\m-2}{4(\m-1)}} \right\|_{L^2_v(\Ot)} \cdot t^{-\frac{1}{2\m}} \| u(t) \|_{\wt{X}} \\
& \lesssim t^{-\frac{\m-1}{2\m}} \cdot t^{-\frac{1}{2\m}} \| u(t) \|_{\wt{X}}.
\end{align*}
These estimates and \eqref{gammaapp1} show \eqref{gammaapp2}.

We are now in position to prove \eqref{ugsp1}.
We set $w_k (t,x) := e^{-i\phi (t,x)} \dx^k \uhp (t,x)$.
By \eqref{uepL2}, \eqref{est:uepp}, and \eqref{gammaapp2}, we have
\begin{align*}
& \dx^k u^+ (t,vt) - i^k \lambda v^{\frac{k}{\m-1}} e^{i \phi (t,vt)} \gamma (t,v) \\
& = \dx^k \uhp (t,vt) - \lambda e^{i\phi (t,vt)} \int_{\R} \dx^k \uhp (t,x) \ol{\Psi_v(t,v)} dx + R_k (t,v) \\
& = e^{i\phi (t,vt)} \lambda \int_{\R} \left( w_k (t,vt) - w_k (t,x) \right) \chi (\lambda (x-vt)) dx + R_k (t,v).
\end{align*}
With a change of variables using $z = \lambda (x-vt)$ and \eqref{uhpL2}, we see that
\begin{align*}
& \int_{\R} \left| \left( w_k (t,vt) - w_k (t,x) \right) \chi (\lambda (x-vt)) \right| dx \\
& \le \lambda^{-1} \int_{\R} \left| \left( w_k (t,vt) - w_k (t, \lambda^{-1} z +vt) \right) \chi (z) \right| dz \\
& = \lambda^{-2} \int_{\R} \left| \int_0^1 \dx w_k (t, vt+(1-\theta) \lambda^{-1} z) d\theta \cdot z \chi (z) \right| dz \\
& \lesssim t^{-\frac{1}{\m-1}} \lambda^{-\frac{3}{2}} (tv)^{\frac{k-\m+2}{\m-1}} \left\| |x|^{-\frac{k-\m+2}{\m-1}} \J_+ \dx^k \uhp (t) \right\|_{L^2} \\
& \lesssim t^{-\frac{k}{\m}} (t^{\frac{\m-1}{\m}} v)^{\frac{k}{\m-1}-\frac{\m-2}{4(\m-1)}} \cdot t^{-\frac{1}{2\m}} \| u(t) \|_{\wt{X}}.
\end{align*}
From \eqref{lambda}, we obtain the $L^{\infty}$-estimate in \eqref{ugsp1}.

A change of variables using$z = \lambda (x-vt)$ and $\wt{v} = vt+(1-\theta) \lambda^{-1} z$, and \eqref{uhpL2} give
\begin{align*}
& \left\| t^{\frac{k}{\m}+\frac{\m+1}{2\m}} (t^{\frac{\m-1}{\m}} v)^{-\frac{k}{\m-1}+\frac{\m-2}{2(\m-1)}} \lambda \int_{\R} |w_k (t,vt) - w_k (t,x) | \chi (\lambda (x-vt)) dx \right\|_{L^2_v (\Ot)} \\
& \le t^{\frac{3\m^2-6\m+1}{2\m (\m-1)}} \int_{\R} \int _0^1 \left\| v^{-\frac{k-\m+2}{\m-1}} \left( \J_+ \dx^k \uhp \right) (t, vt+(1-\theta) \lambda^{-1} z) \right\|_{L^2_v (\Ot)} |z \chi (z)| d\theta dz  \\
& \le t^{\frac{k}{\m-1}-\frac{1}{2\m}} \left\| \wt{v}^{-\frac{k-\m+2}{\m-1}} \left( \J_+ \dx^k \uhp \right) (t,\wt{v}) \right\|_{L^2_{\wt{v}}} \\
& \lesssim t^{-\frac{1}{2\m}} \| u(t) \|_{\wt{X}},
\end{align*}
which shows the $L^2$-estimate in \eqref{ugsp1}.

Next, we consider the estimates in the frequency spaces.
Because
\begin{align*}
\left| \int_{-\infty}^0 \lambda^{-1} \chi_1 (\lambda^{-1} (\xi-\xi_v), \lambda^{-1}\xi_v) d\xi \right|
&= \left| \int_{-\infty}^{-\lambda^{-1} \xi_v} \chi_1 (\zeta, \lambda^{-1} \xi_v) d\zeta \right| \\
&\le C_2 (t^{\frac{\m-1}{\m}} v)^{-\frac{\m}{2(\m-1)}},
\end{align*}
Proposition \ref{prop:est|u|}, Lemma \ref{lem:freq_psi}, \eqref{Psi_L1}, and \eqref{estu-} yield
\begin{align*}
& \left| \wh{u}(t,\xi_v) - \sqrt{\m-1} e^{-\frac{1}{\m} it\xi_v^{\m}} \gamma (t,v) \right| \\
& \le \left| \int_{\R_+} \left( \wh{u}(t,\xi_v) e^{\frac{1}{\m} it\xi_v^{\m}} - \wh{u}(t,\xi) e^{\frac{1}{\m} it\xi^{\m}} \right) \lambda^{-1} \ol{\chi_1 (\lambda^{-1} (\xi-\xi_v), \lambda^{-1}\xi_v)} d\xi \right| \\
& \quad + (C_1+C_2) (t^{\frac{\m-1}{\m}} v)^{-\frac{\m}{2(\m-1)}} | \wh{u} (t,\xi_v)| + C (t^{\frac{\m-1}{\m}} v)^{-1} \cdot t^{-\frac{1}{2\m}} \| u(t) \|_{\wt{X}}.
\end{align*}
With a change of variables using $\zeta = \lambda^{-1} (\xi-\xi_v)$, we have
\begin{equation} \label{eq:ftuinf}
\begin{aligned}
& \left| \int_{\R_+} \left( \wh{u}(t,\xi_v) e^{\frac{1}{\m} it\xi_v^{\m}} - \wh{u}(t,\xi) e^{\frac{1}{\m} it\xi^{\m}} \right) \lambda^{-1} \ol{\chi_1 (\lambda^{-1} (\xi-\xi_v), \lambda^{-1}\xi_v)} d\xi \right| \\
& \le \int_{\R} |\xi-\xi_v| \int_0^1 \left| \wh{\J u} (t,\theta (\xi_v-\xi) + \xi) \right| d\theta \lambda^{-1} |\chi_1 (\lambda^{-1} (\xi-\xi_v), \lambda^{-1}\xi_v)| d\xi \\
& = \lambda \int_{\R} \int_0^1 \left| \wh{\J u} (t, \xi_v + \lambda \zeta (1-\theta)) \right| d\theta |\zeta \chi_1 (\zeta, \lambda^{-1} \xi_v)| d\zeta.
\end{aligned}
\end{equation}
Because
\[
|\wh{u}(t,\xi_v)| \le \left| \wh{u}(t,\xi_v) - \sqrt{\m-1} e^{-\frac{1}{\m} it\xi_v^{\m}} \gamma (t,v) \right| + \sqrt{\m-1} |\gamma (t,v)|
\]
and $\chi_1 (\cdot ,\al) \in \mathcal{S} (\R)$ for $\al \ge 1$, by Proposition \ref{prop:est|u|}, \eqref{lambda}, and \eqref{constast}, we have
\begin{align*}
& \Big| \wh{u}(t,\xi_v) - \sqrt{\m-1} e^{-\frac{1}{\m} it\xi_v^{\m}} \gamma (t,v) \Big| \\
& \lesssim \lambda^{\frac{1}{2}} \| \J u (t) \|_{L^2} + (t^{\frac{\m-1}{\m}} v)^{-\frac{\m}{2(\m-1)}} | \gamma (t,v) | +  (t^{\frac{\m-1}{\m}} v)^{-1} \cdot t^{-\frac{1}{2\m}} \| u(t) \|_{\wt{X}}.\\
& \lesssim (t^{\frac{\m-1}{\m}} v)^{-\frac{\m-2}{4(\m-1)}} \cdot t^{-\frac{1}{2\m}} \| u(t) \|_{\wt{X}},
\end{align*}
which shows the $L^{\infty}$-estimate in \eqref{ugfs}.

For the $L^2$-estimate in the frequency space, we change variables using $\wt{\mathrm{v}} = \xi_v + \lambda \zeta (1-\theta)$.
Because
\[
\frac{d \wt{\mathrm{v}}}{dv} = \frac{1}{\m-1} v^{-\frac{\m-2}{\m-1}} \left\{ 1- \frac{\m-2}{2} \zeta (1-\theta) t^{-\frac{1}{2}} v^{-\frac{\m}{2(\m-1)}} \right\},
\]
\eqref{eq:ftuinf} yields
\begin{align*}
& \left\| \wh{u}(t,\xi_v) - \sqrt{\m-1} e^{-\frac{1}{\m} it\xi_v^{\m}} \gamma (t,v) \right\|_{L^2_v (\Ot)} \\
& \lesssim \int_0^1 \left\| \lambda \int_{\R} \left| \wh{\J u} (t, \xi_v + \lambda \zeta (1-\theta)) \right| \left| \zeta \chi_1 (\zeta, \lambda^{-1} \xi_v) \right| d\zeta \right\|_{L^2_v (\Ot)}  d\theta\\
& \quad + \left\| (t^{\frac{\m-1}{\m}} v)^{-\frac{\m}{2(\m-1)}} \wh{u} (t,\xi_v) \right\|_{L^2_v(\Ot)} + t^{-\frac{\m-1}{2\m}} \cdot t^{-\frac{1}{2\m}} \| u(t) \|_{\wt{X}} \\
& \lesssim t^{-\frac{1}{2}} \left\| (\J u) (t,\wt{\mathrm{v}}) \right\|_{L^2_{\wt{\mathrm{v}}}} 
+ t^{-\frac{\m-1}{2\m}} \cdot t^{-\frac{1}{2\m}} \| u(t) \|_{\wt{X}} \\
& \lesssim t^{-\frac{\m-1}{2\m}} \cdot t^{-\frac{1}{2\m}} \| u(t) \|_{\wt{X}},
\end{align*}
which concludes the $L^2$-estimate in \eqref{ugfs}
\end{proof}

\section{Proof of the main theorem} \label{S:proof}

We show the following estimate for $\dot{\gamma}$.

\begin{prop} \label{prop:gamma_decay}
Let $u$ be a solution to \eqref{hKdV} that satisfies \eqref{est:u_infty0}.
Then, for $t \ge 1$, we have
\[
\left\| t (t^{\frac{\m-1}{\m}} v)^{\frac{\m-2}{4(\m-1)}} \dot{\gamma} (t) \right\|_{L_v^{\infty} (\Ot)} +\left\| t^{\frac{3\m-1}{2\m}} \dot{\gamma} (t) \right\|_{L_v^2 (\Ot)}
\lesssim \eps,
\]
where the implicit constant is independent of $D$, $T$, and $\eps$.
\end{prop}

\begin{proof}
A direct calculation yields
\begin{align*}
\dot{\gamma}(t,v)
& = \int_{\R} \left( \Lc u \cdot \ol{\Psi_v} + u \ol{\Lc \Psi_v} \right) (t,x) dx \\
& = - \int_{\R} \left( u^{\m} \dx \ol{\Psi_v} \right) (t,x) dx + \int_{\R} \left( u \ol{\Lc \Psi_v} \right)  (t,x) dx.
\end{align*}
The bootstrap assumption \eqref{est:u_infty0} yields
\[
\left| \int_{\R} \left( u^{\m} \dx \ol{\Psi_v} \right) (t,x) dx \right|
\lesssim t^{-1} (t^{\frac{\m-1}{\m}} v)^{-\frac{\m (\m-3)}{2(\m-1)}} (D\eps)^{\m}
\le t^{-1} (t^{\frac{\m-1}{\m}} v)^{-\frac{\m (\m-3)}{2(\m-1)}} \eps.
\]
Moreover, from \eqref{eq:LPsi}, \eqref{Psi_L1}, Lemma \ref{LPsi1}, and Proposition \ref{prop:est|u|}, we have
\begin{align*}
& \left| \int_{\R} \left( u \ol{\Lc \Psi_v} \right) (t,x) dx \right| \\
& \lesssim \frac{1}{t \lambda} \left| \int_{\R} \left( e^{-i\phi} \uhp \ol{\dx \wt{\chi}} \right) (t,x) dx \right| + \frac{1}{t \lambda} \left| \int_{\R} \left( e^{-i\phi} \left( \ol{\uhp} + \ue \right) \ol{\dx \wt{\chi}} \right) (t,x) dx \right| \\
& \quad + \int_{\R} \left| (u \dx^{\m} P^- \Psi_v ) (t,x) \right| dx
+ t^{-1} (t^{\frac{\m-1}{\m}} v)^{-\frac{\m}{\m-1}} \int_{\R} |u(t,x) \xc (\lambda (x-vt))| dx \\
& \lesssim t^{-\frac{\m^2-\m+1}{\m (\m-1)}} (t^{\frac{\m-1}{\m}} v)^{\frac{\m-2}{2(\m-1)}} \int_{\R} \left| \J_+ \uhp (t,x) \wt{\chi} (t,x) \right| dx \\
& \quad + t^{-1} (t^{\frac{\m-1}{\m}} v)^{-\frac{2\m-1}{2(\m-1)}} \cdot t^{-\frac{1}{2\m}} \| u (t) \|_{\wt{X}}.
\end{align*}
Here, from \eqref{uhpL2}, we obtain
\begin{align*}
& t^{-\frac{\m^2-\m+1}{\m (\m-1)}} (t^{\frac{\m-1}{\m}} v)^{\frac{\m-2}{2(\m-1)}} \int_{\R} \left| \J_+ \uhp (t,x) \wt{\chi} (t,x) \right| dx \\
& \lesssim t^{-\frac{\m+1}{\m}} (t^{\frac{\m-1}{\m}} v)^{-\frac{\m-2}{2(\m-1)}} \left\| |x|^{\frac{\m-2}{\m-1}} \J_+ \uhp (t) \right\|_{L^2} \| \wt{\chi} (t) \|_{L^2} \\
& \lesssim t^{-1} (t^{\frac{\m-1}{\m}} v)^{-\frac{\m-2}{4(\m-1)}} \cdot t^{-\frac{1}{2\m}} \| u(t) \|_{\wt{X}}.
\end{align*}
In addition, we use \eqref{uhpL2} and \eqref{eq:L^2v} to obtain
\begin{align*}
& \left\| t^{-\frac{\m^2-\m+1}{\m (\m-1)}} (t^{\frac{\m-1}{\m}} v)^{\frac{\m-2}{2(\m-1)}} \int_{\R} \left| \J_+ \uhp (t,x) \wt{\chi} (t,x) \right| dx \right\|_{L^2_v (\Ot)} \\
& \lesssim t^{-\frac{3}{2}} \left\| |x|^{\frac{\m-2}{\m-1}} \J_+ \uhp \right\|_{L^2} \\
& \lesssim t^{-\frac{3\m-1}{2\m}} \cdot t^{-\frac{1}{2\m}} \| u(t) \|_{\wt{X}}. \qedhere
\end{align*}
\end{proof}

First, we prove global existence of the solution to \eqref{hKdV}.
From Proposition \ref{prop:WP1} and Lemma \ref{lem:energy}, this is equivalent to showing \eqref{est:u_infty}, that is to say, to closing the bootstrap estimate \eqref{est:u_infty0}.
When $t^{-\frac{1}{\m}} |x| \lesssim 1$, Lemma \ref{lem:Xtilde} and Proposition \ref{prop:est|u|} yield
\[
\left\| \lr{t^{-\frac{1}{\m}}x}^{-\frac{k}{\m-1}+\frac{\m-2}{2(\m-1)}} \dx^k u(t) \right\|_{L^{\infty}(t^{-\frac{1}{\m}} |x| \lesssim 1)} \lesssim \eps t^{-\frac{k+1}{\m}}.
\]
When $t^{-\frac{1}{\m}} |x| \gtrsim 1$, owing to \eqref{ugsp1}, we only need to show that
\[
\| \gamma (t) \|_{L^{\infty}_v(\Ot)} \lesssim \eps,
\]
where the implicit constant is independent of $D$, $T$, and $\eps$.

For $v \ge C_{\ast}$, where $C_{\ast}$ is defined by \eqref{constast}, the Gagliardo-Nirenberg inequality, Proposition \ref{prop:WP1} and Lemma \ref{lem:freq_psi} lead to
\[
| \gamma (1,v) |
\lesssim \| \wh{u} (1) \|_{L^{\infty}}
= \left\| e^{\frac{1}{\m}i |\xi|^{\m-1}\xi} \wh{u} (1) \right\|_{L^{\infty}}
\lesssim \| u(1) \|_{L^2}^{\frac{1}{2}} \| \J u(1) \|_{L^2}^{\frac{1}{2}}
\lesssim \eps.
\]
The fundamental theorem of calculus and Proposition \ref{prop:gamma_decay} yield
\[
\gamma (t,v) = \gamma (1,v) + O \left( \eps (t^{\frac{\m-1}{\m}} v)^{-\frac{\m-2}{4(\m-1)}} \right),
\]
which implies
\[
|\gamma (t,v) | \lesssim \eps
\]
for $t \in [1,T]$.

When $0<v< C_{\ast}$, set $t_0 := (C_{\ast} v^{-1})^{\frac{\m}{\m-1}}>1$.
Note that $v \in \Ot$ for $t \ge t_0$.
Then, Bernstein's inequality, \eqref{Psi_L1}, Proposition \ref{prop:est|u|}, and Lemma \ref{lem:Xtilde} yield
\[
|\gamma (t_0,v) |
\lesssim t_0^{\frac{1}{2\m}} \sum_{\substack{N \in 2^{\mathbb{Z}} \\ N \sim t_0^{-\frac{1}{\m}}}} \| u_N (t_0) \|_{L^2} + \eps
\lesssim \eps.
\]
The fundamental theorem of calculus and Proposition \ref{prop:gamma_decay} lead to
\[
\gamma (t,v) = \gamma (t_0,v) + O \left( \eps \right),
\]
which implies
\[
|\gamma (t,v) | \lesssim \eps
\]
for $t \in [t_0,T]$.
Accordingly, we conclude that \eqref{est:u_infty} holds for any $t \in [1,T]$.

Second, we show the existence of a self-similar solution.
We use the self-similar change of variables \eqref{selfsimilar}.
Let $\rho >0$ be a constant specified later.
We set $\Orz := \{ y \in \R \colon |y| \le C_{\ast} t^{(\m-1)\rho} \}$ and $\con := 4 (\kk C_{\ast})^{\frac{1}{\m-1}}$.
For $k=0,1, \dots, \m-2$, estimates \eqref{eq:selfbound} and \eqref{uepNl2} and Lemmas \ref{lem:energy} and \ref{lem:Xtilde} imply that
\begin{align*}
& \left\| \dt \left( P_{\le \con t^{\rho}} \dy^k U \right) \right\|_{L^{\infty}_y (\Orz)} \\
& \lesssim t^{(k+\frac{1}{2})\rho} \left\| P_{\le \con t^{\rho}} \dt U \right\|_{L^2_y} + t^{-1} \| P_{\con t^{\rho}} \dy^k U \|_{L^{\infty}_y (\Orz)} \\
& \lesssim t^{(k+\frac{3}{2})\rho-\frac{1}{2\m}-1} \| \Lambda u \|_{L^2_x} + t^{(k+\frac{1}{2})\rho+\frac{1}{2\m}-1} \sum _{\substack{N \in 2^{\mathbb{Z}} \\ N \sim \con t^{\rho-\frac{1}{\m}}}} \| \uep_N \|_{L^2} \\
& \lesssim \eps t^{-1+(k+\frac{3}{2})\rho-\min \left( \frac{1}{2\m}-\eps, \m \rho \right)}.
\end{align*}
Furthermore, \eqref{est:ueperr}, \eqref{uepNl2}, and Lemma \ref{lem:Xtilde} yield
\begin{align*}
&
\begin{aligned}
\left\| P_{> \con t^{\rho}} \dy^k U \right\|_{L^{\infty}_y (\Orz)}
& \lesssim t^{\frac{k+1}{\m}} \sum _{\substack{N \in 2^{\mathbb{Z}} \\ N > \con t^{\rho-\frac{1}{\m}}}} N^{k+\frac{1}{2}} \| \uep_N \|_{L^2} \\
& \quad + t^{\frac{k+1}{\m}} \sum _{\substack{N \in 2^{\mathbb{Z}} \\ N > \con t^{\rho-\frac{1}{\m}}}} \left\| \left( 1-\PN \right) |\dx|^{k+\frac{1}{2}} \uep_N \right\|_{L^2} \\
& \lesssim t^{(k-\m+\frac{3}{2}) \rho} \eps,
\end{aligned}
\\
&
\begin{aligned}
\| P_{> \con t^{\rho}} \dx^k U \|_{L^2_y (\Orz)}
& \lesssim t^{\frac{k}{\m}+\frac{1}{2\m}} \sum _{\substack{N \in 2^{\mathbb{Z}} \\ N > \con t^{\rho-\frac{1}{\m}}}} N^k \| \uep_N \|_{L^2} \\
& \quad + t^{\frac{k}{\m}+\frac{1}{2\m}} \sum _{\substack{N \in 2^{\mathbb{Z}} \\ N > \con t^{\rho-\frac{1}{\m}}}} \left\| \left( 1-\PN \right) |\dx|^k \uep_N \right\|_{L^2} \\
& \lesssim t^{(k-\m+1)\rho} \eps.
\end{aligned}
\end{align*}
By setting $\rho := \frac{1}{\m} (\frac{1}{2\m}-\eps)$, there exists $Q =Q(y) \in L^{\infty}_y (\R)$ such that
\begin{equation}
\| \dy^k U(t) - \dy^k Q \|_{L^{\infty}_y (\Orz)} \lesssim \eps t^{(k-\m+\frac{3}{2})\rho}, \label{estQ}
\quad
\| \dx^k U(t) - \dx^k Q \|_{L^2_y (\Orz)} \lesssim \eps t^{(k-\m+1)\rho}.
\end{equation}
By \eqref{est:u_infty} and the first estimate in \eqref{estQ}, we see that
\begin{equation} \label{selfest}
\begin{aligned}
\left\| \lr{\cdot}^{\frac{\m-2}{2(\m-1)}} Q \right\|_{L^{\infty} (\R)}
&\le \lim_{t \rightarrow \infty} \left( t^{\frac{\m-2}{2}\rho} \left\| Q-U(t) \right\|_{L^{\infty} (\Orz)} + \left\| \lr{\cdot}^{\frac{\m-2}{2(\m-1)}} U(t) \right\|_{L^{\infty} (\R)} \right) \\
&\lesssim \eps.
\end{aligned}
\end{equation}
Because Lemma \ref{lem:Xtilde} implies that
\begin{equation} \label{limQ}
\left\| |\dy|^{\m-1} U - y U - \m U^{\m} \right\|_{L^2_y}
= \left\| (\Lambda u) (t,t^{\frac{1}{\m}}y) \right\|_{L^2_y}
\lesssim t^{-\frac{1}{2\m}+\eps},
\end{equation}
by taking the limit as $t \rightarrow \infty$, we have that $Q$ is a solution to \eqref{eq:self}.
Moreover, \eqref{estQ} and the mass conservation law yield
\[
\int_{\R} Q(y) dy = \lim_{t \rightarrow \infty} \int_{\R} U(t,y) dy = \int_{\R} u_0(x) dx.
\]
Therefore, $\self (t,x) := t^{-\frac{1}{\m}} Q(t^{-\frac{1}{\m}} x)$ is a solution to \eqref{hKdV} with $\self (0) = \int_{\R} u_0(x) dx \delta_0$, where  $\delta_0$ denotes the Dirac delta measure concentrated at the origin.

Finally, we prove the asymptotic behavior of the global solution.
The estimates in the self-similar region $\Xz$ follow from \eqref{estQ}.
Moreover, the estimates in the decaying region $\Xn$ are consequences of Lemma \ref{lem:Xtilde}, \eqref{uepL2}, and \eqref{est:uepp}.
Hence, we only need to show the estimates in the oscillatory region $\Xp$.
Proposition \ref{prop:gamma_decay} yields
\begin{equation} \label{gammaW}
\gamma (t,v) = \gamma (1,v) + \wt{R}(t,v),
\end{equation}
where
\[
\left\| (t^{\frac{\m-1}{\m}} v)^{\frac{\m-2}{4(\m-1)}} \wt{R} (t,v) \right\|_{L^{\infty} (\Ot)} + \left\| t^{\frac{\m-1}{2\m}} \wt{R} (t,v) \right\|_{L^2 (\Ot)} 
\lesssim \eps.
\]
Here, we set
\[
W(\xi) := \sqrt{\m-1} \gamma (1, \xi^{\m-1}),
\]
and extend $W$ to $\R$ by defining
\[
W(-\xi) = \ol{W(\xi)}, \quad
W(0) = \int_{\R} u_0(x) dx.
\]
Then, from \eqref{est:u_infty} and \eqref{eq:L^2v}, we see that
\[
\| W \|_{L^{\infty} \cap L^2}
\lesssim \| u(1) \|_{L^{\infty} \cap L^2}
\lesssim \eps.
\]
Proposition \ref{prop:approx} and \eqref{gammaW} show the estimates in $\Xp$.

\appendix

\section{Well-posedness} \label{wp}

In this appendix, we show the local-in-time well-posedness of the Cauchy problem of \eqref{hKdVp} as well as \eqref{hKdV}.
Here, we assume $\p \ge \m \ge 3$.

To show well-posedness for \eqref{hKdV}, we can apply the Fourier restriction norm method.
In fact, Gr\"{u}nrock \cite{Gru10} proved well-posedness for \eqref{hKdV} in $H^s(\R)$ for odd values of $\m \ge 5$ and $s>-\frac{1}{2}$.
On the other hand, because $F$ may not be a polynomial with respect to $u$, some regularity is needed to be well-posed for \eqref{hKdVp}.

We need to introduce some notation.
Let $1 \le p,q \le \infty$ and $T>0$.
Define
\begin{align*}
& \| f \|_{L_x^p L_T^q} := \left( \int_{-\infty}^{\infty} \left( \int_{-T}^T |f(t,x)|^q dt \right)^{\frac{p}{q}} dx \right)^{\frac{1}{p}},
\\
& \| f \|_{L_T^q L_x^p} := \left( \int_{-T}^T \left( \int_{-\infty}^{\infty} |f(t,x)|^p dt \right)^{\frac{q}{p}} dx \right)^{\frac{1}{q}},
\end{align*}
with $T=t$ to indicate the case when $T=\infty$.

The maximal function estimate and the local smoothing estimate are the main tools in the proof (see Theorems 2.5 and 4.1 in \cite{KPV91} respectively).

\begin{thm} \label{thm:linest1}
\[
\| \lp (t) u_0 \|_{L_x^4 L_t^{\infty}} \lesssim \left\| |\dx|^{\frac{1}{4}} u_0 \right\|_{L^2}, \quad
\left\| |\dx|^{\frac{\m-1}{2}} \lp (t) u_0 \right\|_{L_x^{\infty} L_t^2} \lesssim \| u_0 \|_{L^2}.
\]
\end{thm}

Stein's interpolation \cite{Ste56} and Theorem \ref{thm:linest1} yield the following.
\begin{lem} \label{lem:mfls}
For any value of $\al$ where $-\frac{1}{4} \le \al \le \frac{\m-1}{2}$, we have
\[
\| |\dx|^{\al} \lp (t) u_0 \|_{L_x^{\frac{2(2\m-1)}{\m-1-2\al}} L_t^{\frac{2(2\m-1)}{1+4\al}}} \lesssim \| u_0 \|_{L^2}.
\]
\end{lem}

In particular, by setting $\al=0$ or $\al=1$, it follows that
\begin{equation} \label{interpolation1}
\| \lp (t) u_0 \|_{L_x^{\frac{2(2\m-1)}{\m-1}} L_t^{2(2\m-1)}} + \| \dx \lp (t) u_0 \|_{L_x^{\frac{2(2\m-1)}{\m-3}} L_t^{\frac{2(2\m-1)}{5}}} \lesssim \| u_0 \|_{L^2}.
\end{equation}
Moreover, setting $\al = \frac{\p+\m-3}{2(\p-1)}$, we have
\begin{equation} \label{interpolation2}
\| \dx \lp (t) u_0 \|_{L_x^{p_0} L_t^{q_0}}
\lesssim \left\| |\dx|^{\frac{\p-\m+1}{2(\p-1)}} u_0 \right\|_{L^2},
\end{equation}
where $\left( p_0,q_0 \right) := \left( \frac{2(2\m-1)(\p-1)}{(\m-2)(\p-2)},  \frac{2(2\m-1)(\p-1)}{3\p+2\m-7} \right)$.

Our next result is a Sobolev type of estimate (the proof is the same as that of Lemma 3.15 in \cite{KPV93}).

\begin{lem} \label{lem:nonestq}
For any $q \in \R$ with $\max \left( \frac{-\p+2\m-3}{2(2\m-1)(\p-1)}, 0\right) < \frac{1}{q} < \frac{\m-2}{2\m (\p-1)}$, we have
\[
\| g \|_{L_x^{r(q)} L_t^q} \lesssim \left\| |\dx|^{\al (q)} g \right\|_{L_x^{\frac{4q}{q-2}} L_t^q},
\]
where $\frac{1}{r(q)} :=- \frac{\m}{q}+\frac{\m-2}{2(\p-1)}$ and $\al (q) := \frac{2\m-1}{2q}-\frac{\m-2}{2(\p-1)}+\frac{1}{4}$.
\end{lem}

\begin{proof}
The assumption implies that $0<\al (q)<1$.
Fix $t$ now and use the fractional integration in $x$ to obtain the representation
\[
g(t,x) = c_{q} \int_{-\infty}^{\infty} \frac{1}{|x-y|^{1-\al (q)}} \left( |\dx|^{\al (q)}g \right) (t,y) dy.
\]
By Minkowski's inequality, we have
\[
\| g(x) \|_{L_t^q} \lesssim \int_{-\infty}^{\infty} \frac{1}{|x-y|^{1-\al (q)}} \left\| \left( |\dx|^{\al (q)}g \right) (y) \right\|_{L_t^q} dy.
\]
From $\al (q) = \frac{q-2}{4q}-\frac{1}{r(q)}$, Hardy-Littlewood-Sobolev's inequality yields the desired bound.
\end{proof}

Lemma \ref{lem:mfls} with $\al = \frac{2\m-1}{2q}-\frac{1}{4}$ implies
\begin{equation} \label{interpolation3}
\left\| |\dx|^{\al (q)} \lp (t) u_0 \right\|_{L_x^{\frac{4q}{q-2}} L_t^q}
\lesssim \left\| |\dx|^{\frac{\p-\m+1}{2(\p-1)}} u_0 \right\|_{L^2}
\end{equation}
provided that $q \in [2,\infty]$.

Setting $q_1 := \frac{(2\m-1)(\p-1)}{\m-3}$, $q_2 := \frac{2(2\m-1)(\p-1)(\p-2)}{(2\m-5)\p-4\m+9}$, and $q_3 := \frac{2\m \p (\p-1)}{(\m-3)\p+1}$, we define
\begin{align*}
&
\| u \|_{Y_T} := \| u \|_{L_T^{\infty} L_x^2} + \| u \|_{L_x^{\frac{2(2\m-1)}{\m-1}} L_T^{2(2\m-1)}} + \| \dx u \|_{L_x^{\frac{2(2\m-1)}{\m-3}} L_T^{\frac{2(2\m-1)}{5}}},
\\
&
\| u \|_{\wt{Y}_T} := \| \dx u \|_{L_x^{p_0} L_T^{q_0}} + \sum_{j=1}^3 \left\| |\dx|^{\al (q_j)} u \right\|_{L_x^{\frac{4q_j}{q_j-2}} L_T^{q_j}}, \\
&
\| u \|_{Z_T^s} := \| \lr{\dx}^s u \|_{Y_T} + \| \Lambda u \|_{Y_T} + \| u \|_{\wt{Y}_T}, \\
&
Z_T^s := \{ u \in C([-T,T]; H^s(\R)) \colon \| u \|_{Z_T^s} < \infty \}
\end{align*}
for $\p \ge \m \ge 3$, $s \in \R$, and $T>0$.

The linear estimates \eqref{interpolation1}, \eqref{interpolation2}, and \eqref{interpolation3} lead to
\begin{equation} \label{wplin}
\| \lp (t) u_0 \|_{Y_T} \lesssim \| u_0 \|_{L^2}, \quad
\| \lp (t) u_0 \|_{\wt{Y}_T} \lesssim \| u_0 \|_{H^{\frac{\p-\m+1}{2(\p-1)}}}.
\end{equation}

We note that if $u$ is a solution to \eqref{hKdVp}, then
\begin{equation} \label{eq:lambda2}
\Lc \Lambda u
= \left( \m u^{\m-1} + F'(u) \right) \dx \Lambda u + \m F(u) - F'(u) u.
\end{equation}

\begin{prop} \label{prop:WPp}
Let $\p \ge \m \ge 3$ and $\frac{\p-\m+1}{2(\p-1)} \le s < 1$.
If $u_0 \in \wS^s (\R)$, then there exists a value $T>0$ that depends on $\| u_0 \|_{\wS^{\frac{\p-\m+1}{2(\p-1)}}}$ and a unique solution $u \in Z_T^s$ to \eqref{hKdVp} satisfying
\[
\sup _{t \in [-T,T]} \left( \| u(t) \|_{H^s} + \| \Lambda u(t) \|_{L^2} \right) \lesssim \| u_0 \|_{\wS^s}.
\]
Moreover, the flow map $u_0 \in \wS^s (\R) \mapsto u \in Z_T^s$ is locally Lipschitz continuous.
\end{prop}

\begin{proof}
The fractional Leibnitz and chain rules (see Appendix in \cite{KPV93}) and Lemma \ref{lem:nonestq} yield
\begin{equation} \label{wp:nonlest}
\begin{aligned}
&\left\| |\dx|^s ( F'(u) \dx u) \right\|_{L_{T,x}^2} \\
&\lesssim
\left\| |\dx|^s \dx u \right\|_{L_x^{\frac{2(2\m-1)}{\m-3}} L_T^{\frac{2(2\m-1)}{5}}} \left\| |u|^{\p-1} \right\|_{L_x^{\frac{2(2\m-1)}{\m+2}} L_T^{\frac{2\m-1}{\m-3}}} \\
& \quad + \left\| |\dx|^s u \right\|_{L_x^{\frac{2(2\m-1)}{\m-1}} L_T^{2(2\m-1)}} \| \dx u \|_{L_x^{p_0} L_t^{q_0}} \left\| |u|^{\p-2} \right\|_{L_x^{\frac{2(2\m-1)(\p-1)}{2\p+\m-4}} L_t^{\frac{2(2\m-1)(\p-1)}{(2\m-5)\p-4\m+9}}} \\
&\lesssim
\left\| |\dx|^s u \right\|_{Y_T} \left( \left\| u \right\|_{L_x^{r(q_1)} L_T^{q_1}}^{\p-1} + \| \dx u \|_{L_x^{p_0} L_t^{q_0}} \left\| u \right\|_{L_x^{r(q_2)} L_t^{q_2}}^{\p-2} \right) \\
& \lesssim
\left\| |\dx|^s u \right\|_{Y_T} \| u \|_{\wt{Y}_T}^{\p-1}.
\end{aligned}
\end{equation}
In addition, H\"{o}lder's inequality and Lemma \ref{lem:nonestq} imply
\begin{equation} \label{wp:nonlest2}
\| |u|^{\p} \|_{L_{T,x}^2}
= \| u \|_{L_{T,x}^{2\p}}^{\p}
\lesssim T^{\frac{1}{2}-\frac{\p}{q_3}} \| u \|_{L_x^{r(q_3)} L_T^{q_3}}^{\p}
\lesssim T^{\frac{1}{2}-\frac{\p}{q_3}} \| u \|_{\wt{Y}_T}^{\p}.
\end{equation}

We define the operator $K_{u_0} (u)$ by
\[
K_{u_0} (u) = \lp (t) u_0 + \int_0^t \lp (t-t') \dx \left( u^{\m} + F(u) \right) (t') dt'.
\]
By \eqref{eq:lambda2}, the estimates \eqref{wplin}, \eqref{wp:nonlest}, and \eqref{wp:nonlest2} show that
\begin{align*}
&\| K_{u_0} (t) \|_{Z_T^s} \\
&\lesssim \| \lp (t) u_0 \|_{Z_T^s} + \int_{-T}^T \Big( \left\| u^{\m-1} \dx u \right\|_{H_x^s} + \left\| u^{\m-1} \dx \Lambda u \right\|_{L_x^2}  \\ 
& \quad + \left\| F'(u) \dx u \right\|_{H_x^s} + \left\| F'(u) \dx \Lambda u \right\|_{L_x^2} + \left\| |u|^{\p} \right\|_{L_x^2} \Big) (t) dt \\
&\lesssim \| u_0 \|_{H^s} + \| x u_0 \|_{L^2} + T^{\frac{1}{2}} \Big( \left\| \lr{\dx}^s (u^{\m-1} \dx u) \right\|_{L_{T,x}^2} + \left\| u^{\m-1} \dx \Lambda u \right\|_{L_{T,x}^2} \\
& \quad + \left\| \lr{\dx}^s (F'(u) \dx u) \right\|_{L_{T,x}^2} + \left\| F'(u) \dx \Lambda u \right\|_{L_{T,x}^2} + \left\| |u|^{\p} \right\|_{L_{T,x}^2} \Big) \\
& \le C \| u_0 \|_{\wS^s} + C T^{\frac{1}{2}} \| u \|_{Z_T^s} \left( \| u \| _{\wt{Y}_T}^{\m-1} + \| u \| _{\wt{Y}_T}^{\p-1} \right).
\end{align*}
A similar calculation as above yields
\begin{align*}
& \| K_{u_0} (u_1) - K_{u_0} (u_2) \|_{Z_T^s} \\
&\le C T^{\frac{1}{2}} \left( \| u_1 \|_{\wt{Y}_T}^{\m-1} + \| u_2 \|_{\wt{Y}_T}^{\m-1} + \| u_1 \|_{\wt{Y}_T}^{\p-1} + \| u_2 \|_{\wt{Y}_T}^{\p-1} \right) \| u_1-u_2 \|_{Z_T^s}.
\end{align*}
Hence, taking $T>0$ with
\[
10 C T^{\frac{1}{2}} \left\{ \left( 2C \| u_0 \|_{\wS^{\frac{\p-\m+1}{2(\p-1)}}} \right)^{\m-1} + \left( 2C \| u_0 \|_{\wS^{\frac{\p-\m+1}{2(\p-1)}}} \right)^{\p-1} \right\} \le 1,
\]
we find that the mapping $K_{u_0}$ is a contraction on the ball
\[
B(2C \| u_0 \|_{\wS^s}) := \{ u \in Z_T^s \colon \| u \|_{Z_T^s} \le 2C \| u_0 \|_{\wS^s} \}.
\]
Accordingly, there exists a unique solution $u \in B (2C \| u_0 \|_{\wS^s})$ to \eqref{hKdVp}.
\end{proof}

\section{Short-range perturbations} \label{A:hKdVp}

In this appendix, we outline some modifications to Theorem \ref{thm} in the case of short-range perturbations \eqref{hKdVp}.
The main differences appear in the energy estimate.

Let $\p$ be a real number with $\p>\m$.
The existence of a local-in-time solution $u$ with $\lp (-t) u(t) \in \wS^{\frac{\p-1}{2\p}}(\R)$ follows from Proposition \ref{prop:WPp} and $H^{\frac{\p-1}{2\p}}(\R) \hookrightarrow L^{2\p}(\R)$.

We need some modifications in the energy estimate for $\Lambda u$ when $\p \in (\m,\m+\frac{1}{2})$, because the second and third terms on the right hand side of \eqref{eq:lambda2} do not have enough decay.
Set $\dea := \max (\frac{2\m+1-2\p}{2\m},0)$.
Because \eqref{est:u_infty0} yields
\begin{align*}
\| \m F(u) - F'(u) u \|_{L^2}
&\lesssim \| |u|^{\p} \|_{L^2} \\
&\lesssim (D\eps)^{\p} t^{-\frac{\p}{\m}} \left( \int_{|x| \le t^{\frac{1}{\m}}} dx + \int_{|x|\ge t^{\frac{1}{\m}}} \lr{t^{-\frac{1}{\m}}x}^{-\frac{\p (\m-2)}{\m-1}} dx \right)^{\frac{1}{2}} \\
&\lesssim (D\eps)^{\p} t^{-\frac{2\p-1}{2\m}},
\end{align*}
we have
\begin{align*}
&\frac{1}{2} \dt \left( t^{-\dea} \| \Lambda u(t) \|_{L^2} \right)^2 \\
&= -\dea t^{-2\dea-1} \| \Lambda u(t) \|_{L^2}^2  - \frac{\m (\m-1)}{2} t^{-2\dea} \int_{\R} u^{\m-2} \dx u (\Lambda u)^2 dx \\
& \quad - \frac{1}{2} t^{-2\dea} \int_{\R} F''(u) \dx u (\Lambda u)^2 dx + O \left( (D\eps)^{\p} t^{-2\dea-\frac{2\p-1}{2\m}} \| \Lambda u(t) \|_{L^2} \right) \\
&\lesssim  (D\eps)^{\m-1} t^{-1} \left( t^{-\dea} \| \Lambda u(t) \|_{L^2} \right)^2 + (D\eps)^{2\p-\m+1} t^{-2\dea+1-\frac{2\p-1}{\m}}.
\end{align*}
Hence, Gronwall's inequality shows
\[
\| u(t) \|_{X} \lesssim \eps \lr{t}^{\dea+\eps}.
\]
From $\dea < \frac{1}{2\m}$, Lemma \ref{lem:Xtilde} holds true as long as $\| u_0 \|_{\wS^{\frac{\p-1}{2\p}}} \le \eps \ll 1$.

Moreover, \eqref{limQ} is replaced with
\[
\left\| |\dy|^{\m-1} U - y U - \m U^{\m} \right\|_{L^2_y}
= \left\| (\Lambda u) (t,t^{\frac{1}{\m}}y) + t F(u(t,t^{\frac{1}{\m}}y)) \right\|_{L^2_y}
\lesssim t^{-\frac{1}{2\m}+\dea+\eps}.
\]
The remaining arguments in \S \ref{S:proof} are unchanged as long as $\rho = \frac{1}{\m} (\frac{1}{2\m}-\eps)$ is replaced with $\wt{\rho} := \frac{1}{\m}(\frac{1}{2\m}-\dea-\eps)$.

\section*{Acknowledgment}
This work was supported by JSPS KAKENHI Grant number JP16K17624 and Alumni Association ``Wakasatokai'' of Faculty of Engineering, Shinshu University.


\begin{thebibliography}{99}

\bibitem{DeiZho93}
P. Deift and X. Zhou,
\textit{A steepest descent method for oscillatory Riemann-Hilbert problems. Asymptotics for the MKdV equation},
Ann. of Math. (2) 137 (1993), no. 2, 295--368. 

\bibitem{Dod17}
B. Dodson
\textit{Global well-posedness and scattering for the defocusing, mass-critical generalized KdV equation},
Ann. PDE 3 (2017), no. 1, Art. 5, 35 pp.

\bibitem{GPR16}
P. Germain, F. Pusateri, and F. Rousset,
\textit{Asymptotic stability of solitons for mKdV},
Adv. Math. 299 (2016), 272--330.

\bibitem{Gru10}
A. Gr\"{u}nrock,
\textit{On the hierarchies of higher order mKdV and KdV equations},
Cent. Eur. J. Math. 8 (2010), no. 3, 500--536.

\bibitem{HarG16}
B. Harrop-Griffiths,
\textit{Long time behavior of solutions to the mKdV},
Comm. Partial Differential Equations 41 (2016), no. 2, 282--317.

\bibitem{HIT}
B. Harrop-Griffiths, M. Ifrim, and D. Tataru,
\textit{The lifespan of small data solutions to the KP-I},
Int. Math. Res. Not. IMRN 2017, no. 1, 1--28.

\bibitem{HayNau98}
N. Hayashi and P. I. Naumkin,
\textit{Large time asymptotics of solutions to the generalized Korteweg-de Vries equation},
J. Funct. Anal. 159 (1998), no. 1, 110--136.

\bibitem{HayNau99}
N. Hayashi and P. I. Naumkin,
\textit{Large time behavior of solutions for the modified Korteweg-de Vries equation},
Internat. Math. Res. Notices 1999, no. 8, 395--418.

\bibitem{HayNau01}
N. Hayashi and P. I. Naumkin,
\textit{On the modified Korteweg-de Vries equation},
Math. Phys. Anal. Geom. 4 (2001), no. 3, 197--227.

\bibitem{HayNau151}
N. Hayashi and P. I. Naumkin,
\textit{Large time asymptotics for the fourth-order nonlinear Schr\"{o}dinger equation},
J. Differential Equations \textbf{258} (2015), no. 3, 880--905.

\bibitem{HayNau152}
N. Hayashi and P. I. Naumkin,
\textit{Factorization technique for the fourth-order nonlinear Schr\"{o}dinger equation},
Z. Angew. Math. Phys. \textbf{66} (2015), no. 5, 2343--2377.

\bibitem{HayNau153}
N. Hayashi and P. I. Naumkin,
\textit{Global existence and asymptotic behavior of solutions to the fourth-order nonlinear Schr\"{o}dinger equation in the critical case},
Nonlinear Anal. \textbf{116} (2015), 112--131.

\bibitem{HayNau16}
N. Hayashi and P. I. Naumkin,
\textit{Factorization technique for the modified Korteweg–de Vries equation},
SUT J. Math. 52 (2016), no. 1, 49--95.

\bibitem{HirOka16}
H. Hirayama and M. Okamoto,
\textit{Well-posedness and scattering for fourth order nonlinear Schr\"{o}dinger type equations at the scaling critical regularity},
Commun. Pure Appl. Anal. \textbf{15} (2016), no. 3, 831--851.

\bibitem{IfrTat15}
M. Ifrim and D. Tataru,
\textit{Global bounds for the cubic nonlinear Schr\"odinger equation (NLS) in one space dimension},
Nonlinearity 28 (2015), no. 8, 2661--2675.

\bibitem{KatPon88}
T. Kato and G. Ponce,
\textit{Commutator estimates and the Euler and Navier-Stokes equations},
Comm. Pure Appl. Math. 41 (1988), no. 7, 891--907.

\bibitem{KPV89}
C. E. Kenig, G. Ponce, and L. Vega,
\textit{On the (generalized) Korteweg-de Vries equation},
Duke Math. J. 59 (1989), no. 3, 585--610.

\bibitem{KPV91}
C. E. Kenig, G. Ponce, and L. Vega,
\textit{Oscillatory integrals and regularity of dispersive equations},
Indiana Univ. Math. J. 40 (1991), no. 1, 33--69.

\bibitem{KPV93}
C. E. Kenig, G. Ponce, and L. Vega,
\textit{Well-posedness and scattering results for the generalized Korteweg-de Vries equation via the contraction principle},
Comm. Pure Appl. Math. 46 (1993), no. 4, 527--620.

\bibitem{KPV94}
C. E. Kenig, G. Ponce, and L. Vega,
\textit{On the hierarchy of the generalized KdV equations},
Singular limits of dispersive waves (Lyon, 1991), 347--356,
NATO Adv. Sci. Inst. Ser. B Phys., 320, Plenum, New York, 1994. 

\bibitem{KPV00}
C. E. Kenig, G. Ponce, and L. Vega,
\textit{On the concentration of blow up solutions for the generalized KdV equation critical in $L^2$},
Nonlinear wave equations (Providence, RI, 1998), 131--156,
Contemp. Math., \textbf{263}, Amer. Math. Soc., Providence, RI, 2000. 

\bibitem{KocMar12}
H. Koch and J. Marzuola,
\textit{Small data scattering and soliton stability in $\dot{H}^{-1/6}$ for the quartic KdV equation},
Anal. PDE 5 (2012), no. 1, 145--198.

\bibitem{Mer01}
F. Merle,
\textit{Existence of blow-up solutions in the energy space for the critical generalized KdV equation},
J. Amer. Math. Soc. \textbf{14} (2001), no. 3, 555--578.

\bibitem{Oka17}
M. Okamoto,
\textit{Large time asymptotics of solutions to the short-pulse equation},
NoDEA Nonlinear Differential Equations Appl. 24 (2017), no. 4, 24:42.

\bibitem{Okapre}
M. Okamoto,
\textit{Long-time behavior of solutions to the fifth-order modified KdV-type equation},
to appear in Adv. Differential Equations.

\bibitem{Ste56}
E. M. Stein, 
\textit{Interpolation of linear operators},
Trans. Amer. Math. Soc. \textbf{83} (1956), 482--492.

\bibitem{SSS86}
A. Sidi, C. Sulem, and P. L. Sulem,
\textit{On the long time behaviour of a generalized KdV equation},
Acta Appl. Math. 7 (1986), no. 1, 35--47.

\bibitem{Tao07}
T. Tao,
\textit{Scattering for the quartic generalised Korteweg-de Vries equation}
J. Differential Equations 232 (2007), no. 2, 623--651.
\end{thebibliography}
\end{document}